\documentclass[a4paper,12pt]{article}

\usepackage{amsmath,amsthm,amssymb}
\usepackage{enumerate}
\usepackage{graphicx}
\usepackage{titlesec}
\usepackage{empheq}
\usepackage{color,xcolor}
\usepackage{xfrac,bigints}
\usepackage{comment}
\usepackage{subfigure}
\usepackage{caption}
\usepackage{tocbibind}
\usepackage[toc,page]{appendix}
\usepackage[normalem]{ulem}
\usepackage{url}
\usepackage{float}
\usepackage{subcaption}

\usepackage[style=numeric-comp,sorting=nyt,backend=biber, maxbibnames=99]{biblatex}
\allowdisplaybreaks[3]

\usepackage{circuitikz}
\usepackage{tikz}
\usetikzlibrary{calc,patterns,decorations.pathmorphing,positioning}

\allowdisplaybreaks[4]

\def\claim#1{\begin{trivlist}\item[\hskip\labelsep\bf#1]\it}
\def\endclaim{\end{trivlist}}

\newtheorem{Th}{Theorem}[section]
\newtheorem{lemma}[Th]{Lemma}
\newtheorem{remark}[Th]{Remark}

\newtheorem{a priori condition}[Th]{A priori condition}

\titleformat*{\section}{\normalsize\bfseries}
\titleformat*{\subsection}{\normalsize\bfseries}
\renewcommand{\thesection}{{{\normalsize\arabic{section}}}}

\renewcommand{\theequation}{{{\thesection.\arabic{equation}}}}

\DeclareMathOperator*{\esssup}{ess\,sup}

\definecolor{purple}{rgb}{.75, 0, .25}

\makeatletter
\long\def\@makefntext#1{\parindent 1em\noindent
\@hangfrom{\hbox to 1.8em{\hss$^{\@thefnmark}$}}#1}
\makeatother

\numberwithin{equation}{section}

\begin{document}

\title{Uniqueness of dynamic elastography for isotropic standard linear solid viscoelastic media}

\author{Yu Jiang \thanks{School of Mathematics, Shanghai University of Finance and Economics, Shanghai, China (Email: jiang.yu@mail.shufe.edu.cn, ORCID: 0000-0003-4551-9335)}
\and
Ching-Lung Lin\thanks{Department of Mathematics, National Cheng-Kung University, Tainan 701, Taiwan (Email:
cllin2@mail.ncku.edu.tw)}\and
Gen Nakamura\thanks{Department of
Mathematics, Hokkaido University, Sapporo 060-0808, Japan, and Research Center of Mathematics for Social Creativity, Research Institute for Electronic Science, Hokkaido University, Sapporo 060-0812, Japan (Email: gnaka@math.sci.hokudai.ac.jp).
}}

\maketitle

\begin{abstract}
Dynamic elastography is a widely used, safe, convenient, and cost-effective method to aid in medical diagnosis. It visualizes the wave field propagating through living tissues and quantitatively determines the wave propagation speed from the acquired data, thereby enabling the extraction of the viscoelastic properties of in vivo tissues. Notably, this identification process relies on the mathematical modeling of the viscoelastic characteristics of living tissues. When living tissues are simply modeled as isotropic elastic media, J. McLaughlin and J. Yoon established the uniqueness of the identification in \cite{MY} by reasoning that they called the ``shrink and spread argument". Given the realistic viscoelastic nature of biological tissues, generalizing their results by adopting viscoelastic models is of great significance. In this paper, using their reasoning, we prove the uniqueness of identification for two typical viscoelastic media: the isotropic extended Maxwell model and the isotropic extended standard linear solid model. More precisely, we demonstrate that the shear wave speed within a region of interest $\Omega$ can be uniquely determined from a single measurement of the wave field in $\Omega$.

\bigskip


\end{abstract}

\renewcommand{\theequation}{\thesection.\arabic{equation}}

\section{Introduction}\label{introduction}
\label{sec1}
\setcounter{equation}{0}

Dynamic elastography is a biomedical imaging technique that enables quantitative evaluation of tissue stiffness in the region of interest (abbreviated as ROI) $\Omega\subset{\mathbb R}^3$ by generating waves in $\Omega$ and measuring them using focused ultrasound (see \cite{ue1,ue2,ue3}). This method is a non-invasive, simple, rapid, safe, and easy-to-use technique that aids in diagnosis. 
\begin{figure}[ht]
    \centering
  \includegraphics[width=0.8\linewidth]{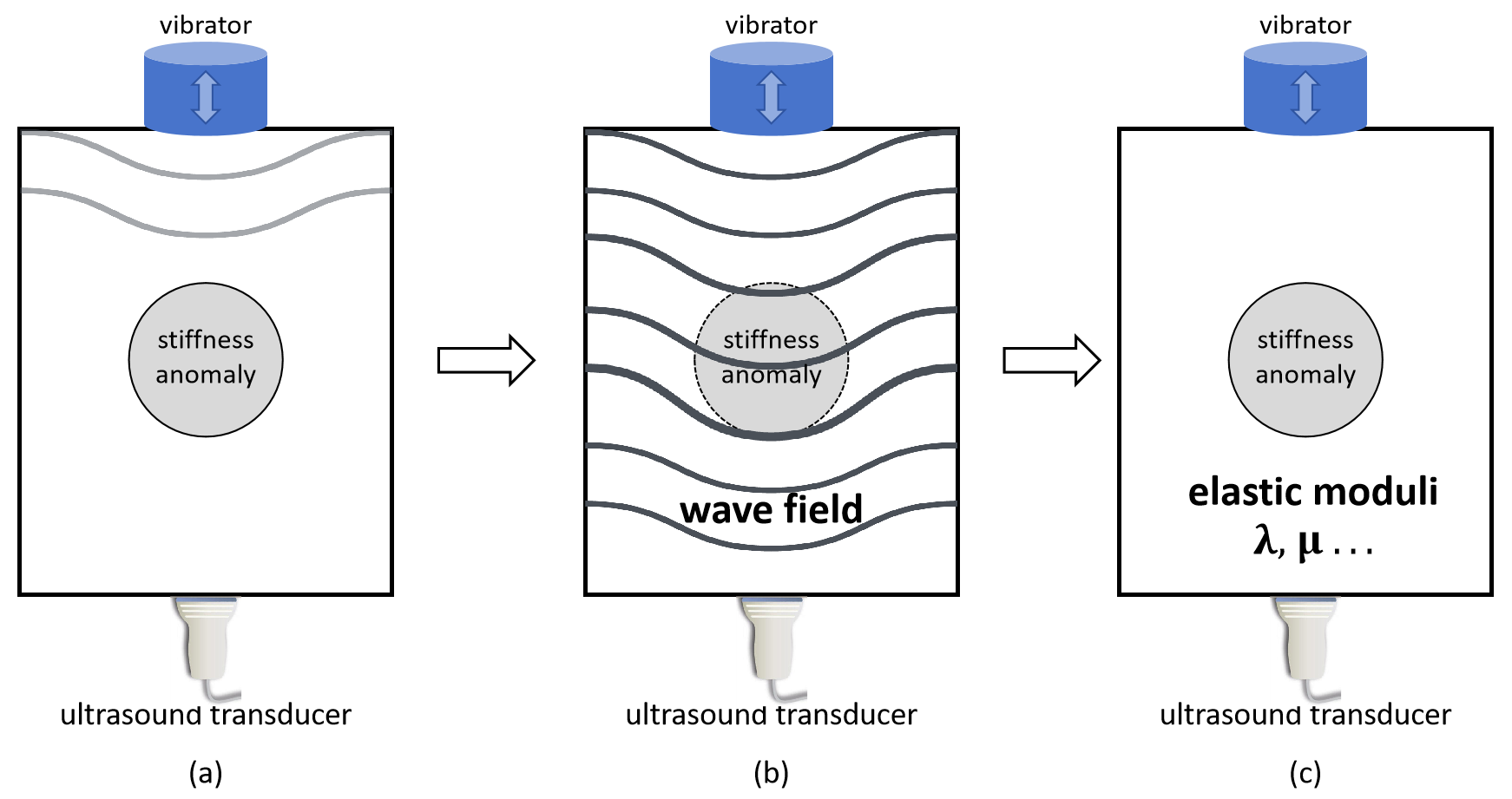}
\caption{Schematic diagram of dynamic elastography.}\label{fig:ue}
\end{figure}

The schematic diagram of dynamic elastography (Figure \ref{fig:ue}), adapted and illustrated according to the relevant literature \cite{ue1,ue2,ue3}, consists of the following: (a) A vibrator placed at the boundary of the ROI generates low-frequency vibrations (30–200 Hz) that propagate deep into the tissue. (b) Ultrasound transducer(s) provide image(s) of the wave field propagating within the ROI. (c) By solving an inverse problem of dynamic elastography, elastic moduli (e.g. $\lambda,\,\mu$) can be reconstructed.

Mathematical modeling of the mechanical properties of tissues is important in dynamic elastography to increase the precision of this method. We model it as the isotropic extended Maxwell model (abbreviated as isotropic EMM) and the isotropic extended standard linear solid model (abbreviated as isotropic ESLS), which are two typical viscoelastic models used by practitioners.

\begin{figure}[ht]  
	\centering 

\def\myscale{1} %1.0
%	\begin{subfigure}[c]{0.4\linewidth}
		\begin{tikzpicture}[scale=\myscale,transform shape,every node/.style={outer sep=0pt},thick,
				spring/.style={thick,decorate,decoration={zigzag,pre length=6*\myscale,
				post length=6*\myscale,segment length=14*\myscale,amplitude=10*\myscale}},
				dampic/.pic={\fill[white] (-0.1,-0.3) rectangle (0.3,0.3);
					\draw (-0.3,0.3) -| (0.3,-0.3) -- (-0.3,-0.3);
				\draw[line width=1mm] (-0.1,-0.3) -- (-0.1,0.3);}
			]
			\foreach \i\j\k\l in {01/1/0/3cm,02/2/-1.5/3cm,99/n/-3.5/3cm}
			%\foreach \i\j\k\l in {1/i/-1.5/2.5cm}
			{
				\node[coordinate] 
				at (0,\k)
				(localcenter\i) {};

				\node (rect) [rectangle,draw,dotted,minimum width=4cm,minimum height=1.2cm,
				] at (localcenter\i) {};
				\node [label={[label distance=1.8cm]5:$M_\j$}] at (localcenter\i) {};

				\node[coordinate,
				right=\l of localcenter\i] (localright\i) {};

				\node[coordinate,
				left=2cm of localcenter\i] (localmidleft\i) {};

				\node[coordinate,
				left=\l of localcenter\i] (localleft\i) {};

				\draw[thick] (localmidleft\i) to (localleft\i);

				\draw[spring] ([yshift=0mm]localmidleft\i.east) coordinate(aux)
				-- (localcenter\i.west|-aux) node[midway,above=1mm]{};

				\draw ([yshift=-0mm]localcenter\i.east) coordinate(aux')
				-- 
				(localright\i.west|-aux') pic[midway]{dampic} node[midway,below=3mm]{};
			}
			\draw [shorten >=0.8cm,shorten <=0.8cm,loosely dotted,thick] (localcenter02) to (localcenter99);
			\draw [shorten >=0.8cm,shorten <=0.8cm,loosely dotted,thick] (localright02) to (localright99);
			\draw [shorten >=0.8cm,shorten <=0.8cm,loosely dotted,thick] (localleft02) to (localleft99);
			\draw [shorten >=0.0cm,shorten <=-0.7cm,thick] (localleft02) to (localleft01);
			\draw [shorten >=0.0cm,shorten <=-0.7cm,thick] (localright02) to (localright01);
			\draw [shorten >=0.0cm,shorten <=1.3cm,thick] (localleft02) to (localleft99);
			\draw [shorten >=0.0cm,shorten <=1.3cm,thick] (localright02) to (localright99);
			\draw[thick] ($(localright02)!-0.2!(localcenter02)$) to (localright02);
			\draw[thick] ($(localleft02)!-0.2!(localcenter02)$) to (localleft02);
		\end{tikzpicture}%
%		\caption{Extended Maxwell model} \label{fig:b-EMM}  
%	\end{subfigure}
	\caption{Extended Maxwell model and its constituent $M_j$  (the Maxwell model), where the zigzag and piston represent a spring and dashpot, respectively. The extended standard linear solid model is obtained under the assumption that some, but not all, constituents lack a dashpot.}\label{fig:EMM}
\end{figure}
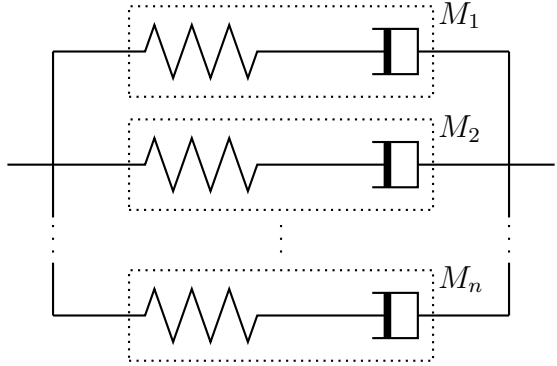

The EMM and ESLS are typical extended spring-dashpot models (abbreviated as ESD) described using springs and dashpots (see, for example, \cite{Christensen, Lakes}), and are used by practitioners to model the mechanical properties of viscoelastic media. As can be seen in Figure \ref{fig:EMM}, these are one-dimensional, isotropic, and homogeneous models. We refer to these models as SD, which we used only for the one dimensional case. Since we want to consider these models extended to multidimensional, anisotropic, and inhomogeneous cases, assuming that the viscosities of dashpots are scalar for simplicity, it is better to briefly explain the extension. For the homogeneous isotropic case, an extension argument can be given by using the volumetric and deviatoric decompositions of strain tensors (see \cite{Flugge}). These decompositions are orthogonal decompositions of the identity matrix acting on the set of 2-rank tensors. In addition, isotropic elastic tensors are commutative with respect to their multiplication. Hence, we can further generalize this argument by using the joint spectral decomposition of the elastic tensors of the springs in the models under the assumption that these tensors are commutative. Furthermore, it is possible to extend this without the assumption using the concept of a tensor network (\cite{ran2020tensor}). Taking this opportunity, we give a mathematical meaning to the ESD  using Figure \ref{fig:EMM}. First of all, for each unit $M_j$, the spectral decomposition of its elastic tensor provides SD over the projected spaces. Between two units with non-commutative elasticity tensors, the transformation rule between these respective units follows the rule under the change of basis between the respective spectral projected spaces of the units. Since this transformation rule is described by a set of 4-rank tensors acting on tensor products of 2-rank tensors by the contraction of 4-rank tensors, we can consider the ESD as a tensor network. Although there is no rigorous homogenization theory for viscoelastic media, it is natural to accept the inhomogeneity of the media simply by assuming that the mechanical properties of the media depend on position. Based on these arguments, it is logically allowed to formally extend homogeneous, isotropic EMM and ESLS in one dimension of space to inhomogeneous, anisotropic EMM and ESLS in two and three dimensions of space.

We first start by describing the background of our study. There is a landmark mathematical work on dynamic elastography by J. McLaughlin and J. Yoon (\cite{MY}). More precisely, they established a mathematical theory to identify the speed of the shear wave based on the shrink and spread argument, using measurements of elastic waves propagating in isotropic elastic living tissue. Mathematically, shrink and spread are the properties of solutions corresponding to finite speed propagation (abbreviated as FSP) and the unique continuation property (abbreviated as UCP), respectively. Subsequently, at the memorial mini-symposium of the 2019 AIP Conference in memory of J. McLaughlin, J. Yoon proposed several problems related to extending the shrink and spread argument to viscoelastic media. 

Based on the background mentioned for this study, this paper aims to solve some problems proposed by J. Yoon. We succeeded in showing the following two things. We give a direct and concise proof of the FPS for the anisotropic EMM and ESLS. The other is to give the UCP for quasi-static isotropic EMM and ESLS. Hence, combining these two things, we succeeded in extending the shrink and spread argument for isotropic EMM and ESLS. In other words, we can uniquely identify the speed of shear wave in the region of interest $\Omega$ by measuring a shear wave in $\Omega$. 

\begin{comment}
Although the shrink-related argument for finite propagation speed (abbreviated by FPS) appeared in \cite{McLaughlin-Yoon,MTY}, the full results have not been published elsewhere to date. We then observed that such an extension requires new ideas, particularly for establishing the spread part, that is, the unique continuation property (abbreviated as UCP), for a special time-dependent elliptic system integrated over the time interval $[0,t]$ for each $0\le t\le T$.
\end{comment}

The remainder of this paper is organized as follows. In Section 2, we present the stress–strain relation for ESLS. In Section 3, we provide a direct common proof of FSP for both anisotropic EMM and ESLS. For the spread argument, we first derive a convenient form of the stress–strain relation for the isotropic ESLS in Section 4. Using this formulation, we establish the spread argument in Section 5 by proving the unique continuation property for the corresponding elliptic system. Finally, using the shrink and spread arguments that we established for isotropic EMM and ESLS, the uniqueness result of the identification speed is given in Section 6, which generalizes the corresponding results given in \cite{MY} for isotropic elastic media. We give conclusion and discussion in Section 7. The Appendix after References presents results on the interior regularity of solutions for isotropic EMM and ESLS, which appear to be new.

Before closing this section, we remark that the results given in this paper are for three space dimensions, but they can also be given for two space dimensions using similar arguments.

\section{Stress-strain relations for EMM and ESLS}\label{stress-strain}\setcounter{equation}{0}
In this section, we derive the relaxation tensor, the total stress and strain relation, for the ESLS. Since the ESLS can be viewed as the EMM with the mentioned assumption, we first derive the relaxation tensor $G(x,t)$ for the EMM and then provide that for the ESLS by giving an interpretation for  $G(x,t)$, where $t\in [0,T]$ with $T>0$, and $x\in\Omega$. Here, we assume that the boundary $\partial\Omega$ of $\Omega$ is Lipschitz smooth and divided into two parts. More precisely, 
\begin{equation*}%\label{two parts}
\left\{
\begin{array}{ll}\Gamma_D,\,\Gamma_N\subset\partial\Omega:\,{\text{non-empty open subsets with Lipschitz boundaries}},\\
\partial\Omega=\overline{\Gamma_D}\cup\overline{\Gamma_N},\,\,\Gamma_D\cap\Gamma_N=\emptyset,
\end{array}
\right.
\end{equation*}
where $\overline \Gamma\subset\partial\Omega$ is the closure of $\Gamma$ for any $\Gamma\subset\partial\Omega$.

For each constituent $M_j$ of Figure \ref{fig:EMM}, the stress-strain relation of a spring described by Hooke's law and that of a dashpot is described as the stress is proportional to the instantaneous change of strain. More precisely, they are given as
\begin{equation}\label{local SS}
\sigma_j^s=C_je_j^s,\quad \sigma_j^d=\eta_j \dot{e}^d,
\end{equation}
where $C_j$ and $\eta_j$ are the elasticity tensor and viscosity, respectively. For a while, we assume that $C_j\in L^\infty(\Omega;{\mathbb R}^{3\times 3\times 3\times 3})$ and $\eta_j\in L^\infty(\Omega; {\mathbb R})$. In addition, the superscripts $s$ and $d$ of the stresses $\sigma_j^s,\,\sigma_j^d$ and the strains $e_j^s,\,e_j^d$ indicate that they are for spring and dashpot, respectively. In addition, ``$\dot{\,\,}$" denotes the time derivative. In addition, we assume the following physically natural assumption for the elasticity tensors and viscosities.

\medskip
\noindent
{\bf Assumption}$\,\,$
For each $1\le j\le n$, let $C_j(x)=((C_j)_{pqrs}(x))$ and $\eta_j(x)$ satisfy the following conditions.  
\begin{itemize}
\item[{\rm (i)}] (full symmetry) $$(C_j)_{pqrs}(x)=(C_j)_{rspq}(x)=(C_j)_{qprs}(x),\,\,\text{a.e. $x\in\Omega$, $1\le p,q,r,s\le 3$}.$$
\item[{\rm (ii)}] (strong convexity) There exists a constant $\delta>0$ such that
the inner product
$C_j(x){\color{magenta}:}(e_{pq}e_{rs})$ of $C_j(x)$ and $(e_{pq}e_{rs})$ satisfies \begin{equation*}%\label{Frobenious}
C_j(x){\color{magenta}:}(e_{pq}e_{rs}):=\sum_{p,q,r,s=1}^3 (C_j(x))_{pqrs}e_{pq}e_{rs}\ge\delta\sum_{p,q=1}^3 e_{pq}^2,\,\,\text{a.e.}\, x\in\Omega,\,\,\text{symmetric}\,(e_{pq}).
\end{equation*}
\item[{\rm(iii)}] (positivity)
\newline There exists a constant $\beta>0$ such that $\eta_j(x)\ge\beta\,\,\text{a.e. $x\in\Omega$}$. 
\end{itemize}

Now, we consider the stress-strain relation for an inhomogeneous anisotropic EMM.
Let $(\sigma_j,e_j)$ and $(\sigma,e)$ be the pair of stress and strain of $M_j$ and that of the total stress and total strain of the EMM. Then, we have
\begin{equation}\label{eq-*2}
\begin{cases}
e_j=e_j^s+e_j^d,\quad\sigma_j=\sigma_j^s=\sigma_j^d,\quad 1\le j\le n,\\
e=e_1=e_2=\cdots=e_n,\quad\sigma=\sigma_1+\sigma_2+\cdots+\sigma_n.
\end{cases}
\end{equation}
The total strain $e$ is given in terms of the displacement vector $u$ as
\begin{equation*}%\label{total strain}
    e=e[u]=\frac{1}{2}(\nabla u+(\nabla u)^{\mathfrak{t}}),
\end{equation*}
where ``$\mathfrak{t}$" denotes the transpose.
For simplicity, denote the viscous strains $e_j^d$, $j=1,2,\dots,n$ by $\phi_j$, $j=1,2,\dots,n$. By combining \eqref{local SS} and \eqref{eq-*2}, we have
\begin{equation}\label{eq-*3}
\left\{\begin{aligned}
& \sigma=\sum_{j=1}^n C_j(e-\phi_j),\\
& \eta_j\,\partial_t\phi_j=C_j(e-\phi_j),\quad j=1,2,\dots,n.
\end{aligned}\right.
\end{equation}
To clarify that $\sigma$ depends on $u$, we write $\sigma=\sigma[u]$.
By assuming $\phi_j(0)=0$, $j=1,2,\dots,n$, we derive
\begin{equation*}%\label{eq-*4}
\sigma(t)=\sum_{j=1}^n C_j\left\{e(t)-\int_0^t e^{-(t-s)\eta_j^{-1}C_j}\eta_j^{-1}C_j e(s)\,ds\right\},
\end{equation*}
which is equivalent to
\begin{equation}\label{VIDM}
\sigma(t)=\sigma_M(t):=\int_0^t G(t-s)\dot e(s)\,ds,\,\,\,G(t):=\sum_{j=1}^n C_j e^{-t\eta_j^{-1}C_j}
\end{equation}
under assumption $e(t)=0$ at $t=0$. We put the subscript $M$ to indicate that $\sigma_M(t)$ is the total stress for the EMM. $G(t)$ is called the relaxation tensor for the EMM.

\medskip
Next, we derive the relaxation tensor for ESLS simply by pointing out the necessary changes in the above derivation of $G(t)$. Since the relation between total stress and strain is independent of the labeling order of $M_j$'s, without loss of generality, we can assume that $M_j$ do not have dashpots so that $\sigma_j^d=e_j^d=\phi_j=0,\,n_1+1\le j\le n$, where $n_1\in{\mathbb N},\,n_1<n$. Then, we need to change the first equation of \eqref{eq-*2} and \eqref{eq-*3} as
\begin{equation*}%\label{ceq-*2}
\left\{
\begin{array}{ll}
e_j=e_j^s+e_j^d\,\,(1\le j\le n_1),\,\,\,&e_j=e_j^s\,\,(n_1+1\le j\le n),\\
\sigma_j=\sigma_j^s=\sigma_j^d\,\,(1\le j\le n_1),\,\,\,&\sigma_j=\sigma_j^s\,\,(n_1+1\le j\le n).
\end{array}
\right.
\end{equation*}
and
\begin{equation*}%\label{ceq-*3}
\left\{
\begin{aligned}%{ll}
&\sigma=\sigma_S(t)=\sum_{j=1}^{n}C_j(e-\phi_j)=\sum_{j=1}^{n_1}C_j(e-\phi_j)+\sum_{j=n_1+1}^n C_je,\\
&\eta_j\partial_t\phi_j=C_j(e-\phi_j),\,\,1\le j\le n_1,
\end{aligned}
\right.
\end{equation*}
respectively. Here, note that the ESLS only has $n_1$ number of viscous strains $\phi_j$, $j=1,2,\cdots n_1$. By direct computations, we have that
\begin{equation}\label{VIDS}
\left\{\begin{aligned}%{ll}
&\sigma(t)=\sigma_S(t)=\int_0^t G(t-s)\dot e(s)\,ds=G(0)e(t)+\int_0^t \dot{G}(t-s)e(s)\,ds,\\
&G(t):=\sum_{j=1}^{n_1} C_je^{-t\eta_j^{-1}C_j}+\sum_{j=n_1+1}^n C_j,
\end{aligned}
\right.
\end{equation}
where the subscript $S$ for $\sigma_S$ indicates that it is for the ESLS, and $G(t)$ is called the relaxation tensor for the ESLS.  

Regarding the dynamic elastography for anisotropic EMM and ESLS, there exists a respective unique solution $u\in C^2([0,T]; H^1(\Omega))$ with its third derivative $t$ in $C^0([0,T]; L^2(\Omega))$ to the initial boundary value problem with mixed-type boundary conditions, which satisfies
\begin{equation}\label{IBP}
\left\{
\begin{array}{ll}
\rho\partial_t^2 u-\text{div}\,\sigma[u]=0,\\
u=g\in C^2([0,T];\dot{H}^{1/2}(\overline{\Gamma_D}))\,\,\text{over $\Gamma_D$},\,\,\,\sigma[u]\nu=0\,\,\text{over $\Gamma_N$},\\
(u,\partial_t u)=(u^0,v^0)\,\,\text{at $t=0$}
\end{array}
\right.
\end{equation}
for each corresponding $\sigma[u]=\sigma_M(t),\,\sigma_S(t)$, if $g, u^0,\,v^0$ satisfies the compatibility condition of order 2 (see the references \cite{Daf, DLN,DKLN} cited in order according to their importance). Here, $\nu$ is the outer unit normal of $\partial\Omega$, and $\dot{H}^{1/2}(\overline{\Gamma_D})$ is the set of elements in $H^{1/2}(\partial\Omega)$ supported in $\overline{\Gamma_D}$. We have omitted the fact that the first equation of \eqref{IBP} is considered in $\Omega\times(0,T)$. Further, we remark that the existence proof of the solution for \eqref{IBP} can be done for each small subsequent interval of the time interval
starting from $[0,t_1]$ with small $t_1$. For instance, in the small interval next to $[0,t_1]$, $\int_0^t\dot{G}(t-s)e(s)ds$ and the initial time $t=0$ have to be replaced by $\int_{t_1}^t\dot{G}(t-s)e(s)ds$ and $t=t_1$, respectively. Hence
$G(0)+\int_{t_1}^t\dot{G}(t-s)\,ds$ for $t_1\le t\le t_2$ with small $t_2-t_1$ acting to $e$ as $\big((G(0)+\int_{t_1}^t\dot{G}(t-s)\,ds)e\big)(t):=G(0)e(t)+\int_{t_1}^t\dot{G}(t-s)e(s)\,ds$ satisfies the strong convexity condition. Repetition of this argument finishes the proof. Then, the time regularity of the solution can be transformed to the interior space regularity of the solution by the standard interior regularity argument for elliptic systems. This argument works not only for the EMM and ESLS but also for general linear viscoelasticity systems discussed in \cite{Daf} and its modification to apply the discussion to the mixed-type boundary conditions \cite{DLN}. Based on this remark, we can have $u\in C^2([0,T]:H^1(\Omega))$ if the following conditions are satisfied: 
\begin{enumerate}
    \item[(H1)] $g\in C^2([0,T]; \dot{H}^r(\overline{\Gamma_D}))$ with $r=5/2+\epsilon,\,0<\epsilon\ll 1$ such that $ g=\tilde{g}|_{\Gamma_D\times[0,T]}$ for a $\tilde g\in C^2([0,T]; H^3(\Omega))$ satisfying $\tilde g=\partial_t \tilde g=0$ at $t=0$.
    \item[(H2)]  $\{\tilde{u}|_{t=0}=0,\,\partial_t\tilde{u}|_{t=0}=0,\, f\}$ satisfies the compatibility condition of order 2 for the abstract Cauchy problem for the equation $$\rho\partial_t^2\tilde u=(\mathcal{L}\tilde u)(t)+f(t),$$
where
$$\tilde u:=u-\tilde g,\,
f:=-\rho\partial_t^2\tilde g+\mathcal{L}\tilde g,$$
with the operator $\mathcal{L}$ given as $$(\mathcal{L}\tilde g)(t):=\nabla \cdot \Big (G(0)e[\tilde g](t)+\int_0^t\dot{G}(t-s)e[\tilde g](s)\,ds\Big )$$
including the homogeneous mixed type boundary condition.
\end{enumerate}

\par
From the derivation of the relaxation tensors, \eqref{IBP} is equivalent to
\begin{equation}\label{equiv-IBP}
\left\{
\begin{array}{ll}
\rho\,\partial_t^2 u-\text{div}\,\sigma[u,\phi]=0,\\
\eta_j\partial_t\phi_j-\sigma_j[u,\phi_j]=0,\,\,1\le j\le n_1,\\
u=g\,\,\text{over $\Gamma_D$},\,\,\sigma[u,\phi]\nu=0 \,\,\,\text{over $\Gamma_N$},\\
(u,\partial_t u, e[u], \phi)=(u^0,v^0,0,0),
\end{array}
\right.
\end{equation}
where for our convenience, we have put $\phi:=(\phi_1,\cdots,\phi_n)=(\phi_1,\cdots,\phi_{n_1},0,\cdots,0)$ with $n_1=n$ and $1\le n_1<n$ for EMM and ESLS, respectively. Here, $\sigma[u,\phi]:=\sum_{j=1}^n\sigma_j[u,\phi_j]$ and each $\sigma_j[u,\phi_j]:=C_j(e[u]-\phi_j)$.

\section{FSP for EMM and ESLS}\label{EMM}\setcounter{equation}{0}

In this section, we provide a direct and concise common proof of the FSP for both the EMM and the ESLS. Before our study, there was a systematic study of the FSP for Boltzmann-type integrodifferential viscoelastic systems in \cite{McLaughlin-Yoon} based on a density formulation. Our results are in line with their results for the EMM and the ESLS. To begin with, consider the solution $u$ of \eqref{equiv-IBP}. Put $v:=\partial_t u$ and $\psi=(\psi_1,\cdots,\psi_n)=e[u]\tilde I-\phi$ with a $3n\times3n$ identity matrix $\tilde I$. Then, we have
\begin{equation}\label{psi_eqs}
\left\{
\begin{aligned}
&\partial_t\psi_j=-\eta_j^{-1}C_j\psi_j+e[v]\,\,\,(1\le j\le n_1),\\
\\
&\partial_t\psi_j=e[v]\,\,\,(n_1+1\le j\le 
n)
\end{aligned}
\right.
\end{equation}
with $n_1$ given after \eqref{equiv-IBP}.

Now, let $D_\alpha(t):=\{x\in{\mathbb R}^3: |x-x_0|\le R-\alpha t\}\subset\Omega$ {with $\alpha>0$} which will be specified later. Define $E_\alpha(v,\psi)(t)$ for $t\ge0$ so that $D_\alpha(t)\not=\emptyset$ by
\begin{equation*}%\label{f3.2}
\begin{aligned}
E_{\alpha}(v,\psi)(t):=&\int_{D_{\alpha}(t)}\Big\{\frac{\rho}{2}|v|^2+\frac{1}{2}\sum_{j=1}^n (C_j\psi_j,\psi_j)\Big\} dt dx\\
=&\frac{1}{2}\int_{D_{\alpha}(t)}\Big\{\rho|\partial_t u|^2+\sum_{j=1}^n\sigma_j[u,\phi_j]:(e[u]-\phi_j)\Big\} dt dx
\end{aligned}
\end{equation*}
with $``:"$ on the right side denoting the inner product of Frobenius. Here and hereafter, we suppress the integration notation at the end of integrals to avoid introducing new notation as much as possible. Also, we abuse the notations $(\,\,,\,)$ and $|\,\,\,\,|$ to denote the Frobenius inner product and the associated norm, respectively, unless they are clear from the context. 

\medskip
Then, we prepare the following lemma.

\medskip\noindent

\begin{lemma}\label{f_jem3.1}
\begin{equation*}%\label{f3.3}
\begin{aligned}
\frac{d}{dt}E_{\alpha}(v,\psi)(t)=&\int_{L_{\alpha}(t)}-\frac{\alpha}{2}\{\rho|v|^2+({C}\psi,\psi)\} +\int_{L_{\alpha}(t)}\big((\sum_{j=1}^n C_j\psi)\nu,\,v\big)\\
&-\int_{D_{\alpha}(t)}\sum_{j=1}^{n_1}\eta_j\big|e[v]-\partial_t\psi_j\big|^2.
\end{aligned}
\end{equation*}
\end{lemma}
\begin{proof}
From the first equation of \eqref{equiv-IBP}, we have the following.
\begin{equation}\label{f3.2}
\begin{aligned}
\int_{D_{\alpha}(t)}\rho(\partial_tv,v)&=\int_{D_{\alpha}(t)}(\nabla \cdot \sum_{j=1}^n C_j\psi_j,v)\\
&=-\sum_{j=1}^n\int_{D_{\alpha}(t)}(C_j\psi_j,e[v])+\int_{L_{\alpha}(t)}\big((\sum_{j=1}^n C_j\psi)\nu,\,v\big).
\end{aligned}
\end{equation}
Also, from  \eqref{psi_eqs}, we have 
\begin{equation}\label{f3.3}
\left\{
\begin{array}{ll}
(C_j\psi_j,\partial_t\psi_j)=-\eta_j|\partial_t\psi_j|^2+(\eta_je[v],\partial_t\psi_j)\,\,(1\le j\le n_1),\\
\\(C_j\psi_j,\partial_t\psi_j)=(C_j\psi_j,e[v])\,\,(n_1+1\le j\le n).
\end{array}
\right.
\end{equation}
By \eqref{f3.2} and \eqref{f3.3}, direct computations yield
\begin{align*}
&\frac{d}{dt}E_{\alpha}(v,\psi)(t)\\
=&\int_{D_{\alpha}(t)}\Big \{\rho(\partial_tv,v)+\sum_{j=1}^n(C_j\psi_j,\partial_t\psi_j)\Big \}-\int_{L_{\alpha}(t)}\frac{\alpha}{2}\Big\{\rho|v|^2+\sum_{j=1}^n(C_j\psi_j,\psi_j)\Big\}\\
=&-\int_{D_{\alpha(t)}}\Big\{\sum_{j=1}^n(C_j\psi_j,e[v])+\sum_{j=1}^{n_1}\eta_j|\partial_t\psi_j|^2-\sum_{j=1}^{n_1}(\eta_je[v],\partial_t\psi_j)-\sum_{i=n_1+1}^n(C_j\psi_j,e[v])\Big\}\\
&-\int_{L_{\alpha}(t)}\frac{\alpha}{2}\Big\{\rho|v|^2+\sum_{j=1}^n(C_j\psi_j,\psi_j)\Big\}+\int_{L_{\alpha}(t)}\Big((\sum_{j=1}^n C_j\psi_j)\nu,\,v\Big)\\
=&-\int_{D_{\alpha(t)}}\Big\{\sum_{j=1}^{n_1}\eta_j|\partial_t\psi_j|^2+ \sum_{j=1}^{n_1}2\eta_j(e[v],\partial_t\psi_j)-\sum_{j=1}^{n_1}\eta_j|e[v]|^2\Big\}\\
&-\int_{L_{\alpha}(t)}\frac{\alpha}{2}\Big\{\rho|v|^2+\sum_{j=1}^n(C_j\psi_j,\psi_j)\Big\}+\int_{L_{\alpha}(t)}\Big((\sum_{j=1}^n C_j\psi_j)\nu,\,v\Big)\\
=&-\int_{D_{\alpha(t)}}\sum_{j=1}^{n_1}\eta_j|e[v]-\partial_t\psi_j|^2-\int_{L_{\alpha}(t)}\frac{\alpha}{2}\Big\{\rho|v|^2+\sum_{j=1}^n(C_j\psi_j,\psi_j)\Big\}\\
&+\int_{L_{\alpha}(t)}\Big((\sum_{j=1}^n C_j\psi_j)\nu,\,v\Big).
\end{align*}
This ends the proof.
\end{proof}

\medskip
Next, let us give the FSP common for both the EMM and the ESLS. We first start by defining the norms $\Vert C_j\Vert_{x}$ and $\Vert Z\Vert_x$ for $\text{a.e}\,\, x\in B_R(x_0)$ by
\begin{equation}\label{norms Cx_Zx}
\Vert C_j\Vert_{x}:=\sup_{|\zeta_j|\neq 0}\frac{|(C_j\zeta_j):\zeta_j|}{|\zeta_j|^2},\,\,\, \Vert Z\Vert_{x}:=\sum_{j=1}^n\Vert C_j\Vert_{x},
\end{equation}
where $\zeta_j,\,j=1,2,\cdots,n$ are 2-rank tensors, and $B_R(x_0):=\{x\in \Omega : |x-x_0|\leq R\} \subset \Omega$.

\medskip
Now we are ready to state the FSP as follows.

\begin{Th}\label{f_thm3.2} The unspecified $\alpha$ can be specified as 
\begin{equation*}%\label{f3.8}
\begin{aligned}
\alpha=\max_{x\in B_R(x_0)}\sqrt{\Vert Z\Vert_{x}/\rho(x)}
\end{aligned}
\end{equation*}
to give the FSP stated as 
\begin{equation}\label{def_FSP}
u=0\,\,\,\text{in}\,\,\, \bigcup_{0<t<s}B_{R-\alpha t}(x_0)\,\,\,\text{for any}\,\,\, s\in(0, R/\alpha)
\end{equation}
whenever $B_R(x_0)\subset\Omega$ and $u=v=0$ in $B_R(x_0)$ at $t=0$.
\end{Th}

\begin{proof}
Write $(C_j\psi_j,\psi_j)$ as
\begin{equation*}%\label{f3.10}
\begin{aligned}
(C_j\psi_j,\psi_j)= C_j^{1/2}\psi_j : C_j^{1/2}\psi_j.
\end{aligned}
\end{equation*}
\noindent
By direct computations, we have that
\begin{equation}\label{f3.6}
\begin{aligned}
|\big( C_j\psi_j)\nu,\,v\big)|\leq |C_j\psi_j||v|= |v|\big(|C_j\psi_j|^2\big)^{1/2}.
\end{aligned}
\end{equation}
We further have that
\begin{equation}\label{f3.7}
\begin{aligned}
|C_j\psi_j|^2\leq (C_j^{1/2}(C_j^{1/2}\psi_j),C_j^{1/2}(C_j^{1/2}\psi_j))
\leq \Vert C_j\Vert_{x}\, (C_j\psi_j,\psi_j).
\end{aligned}
\end{equation}
Then, combining \eqref{f3.6} and \eqref{f3.7}, we have
\begin{equation}\label{f3.8}
\begin{aligned}
|\big(( C_j\psi_j)\nu,\,v\big)|\leq |v||Y_j|\Vert C_j\Vert^{\frac12}_x\,\,\,\text{a.e. $x\in B_R(x_0)$},
\end{aligned}
\end{equation}
where $|Y_j|^2:=(C_j\psi_j,\psi_j)$.
Furthermore, by Lemma \ref{f_jem3.1}, \eqref{norms Cx_Zx} and \eqref{f3.8}, we have
\begin{equation}\label{f3.9}
\begin{aligned}
\frac{d}{dt} E_\alpha(t)(v,\psi)&\leq \int_{L_{\alpha}(t)}-\frac{\alpha}{2}\Big\{\rho|v|^2+\sum_{j=1}^n(C_j\psi_j,\psi_j)\Big\}+\int_{L_{\alpha}(t)}\Big((\sum_{j=1}^n C_j\psi_j)\nu,\,v\Big)\\
&=\sum_{j=1}^n\int_{L_{\alpha}(t)}\Big\{-\frac{\alpha}{2}\{\rho|v|^2\frac{\Vert C_j\Vert_{x}}{\Vert Z\Vert_{x}}+(C_j\psi_j,\psi_j)\}+\big(( C_j\psi_j)\nu,\,v\big)\Big\}\\
&\le\sum_{j=1}^n\int_{L_{\alpha}(t)}\Big\{-\{\frac{\alpha}{2}\left(\frac{1}{\alpha^2}\Vert C_j\Vert_{x}|v|^2+|Y_j|^2\right)\}+|v||Y_j|\Vert C_j\Vert^{\frac12}_x\Big\}\\
&=\sum_{j=1}^n\int_{L_{\alpha}(t)}\Big\{-\{\frac{1}{2}\left(\frac{1}{\alpha}\Vert C_j\Vert_{x}|v|^2+\alpha|Y_j|^2\right)\}+|v||Y_j|\Vert C_j\Vert^{\frac12}_x\Big\}\\
&=-\frac{1}{2}\sum_{j=1}^n 
\int_{L_{\alpha}(t)}\left(\alpha^{-1/2}\Vert C_j\Vert^{1/2}_x|v|-\alpha^{1/2}|Y_j|\right)^2\\
&\leq 0. 
\end{aligned}
\end{equation}
Since \eqref{f3.9} implies
\begin{equation*}%\label{f3.15}
E_{\alpha}(v,\psi)(t)\leq E_{\alpha}(v,\psi)(0)   
\end{equation*}
for $t\geq 0$ such that $D_\alpha(t)\not=\emptyset$,
we immediately have \eqref{def_FSP}.
\end{proof}

\section{Convenient form of stress-strain relations for isotropic EMM and ESLS}\label{I-EMM}\setcounter{equation}{0}

In this section, we derive a convenient form of the relaxation tensor for isotropic EMM and ESLS, which is essential for the arguments in the next section. To begin with, let the elasticity tensor $C_j=((C_j)_{pqrs})$ be
\begin{equation*}%
\begin{aligned}
(C_j)_{pqrs}={\lambda}_j\delta_{pq}\delta_{rs}+{\mu}_j(\delta_{qs}
\delta_{pr}+\delta_{ps}\delta_{qr}),\,\,\,1\le p,\,q,\,r,\,s\le 3
\end{aligned}
\end{equation*}
with the Kronecker delta $\delta_{pq}$, where $\lambda_j,\,\mu_j\in L^\infty(\Omega)$ $(1\le j\le n)$ are the Lam\'e moduli that satisfy the strong convexity condition
\begin{equation}\label{iso_s-convex}
3\lambda_j+2\mu_j\ge\tilde\delta,\,\,\,\mu_j\ge\tilde\delta,\,\,\, j=1,\cdots, n,\,\,\,\text{a.e. in $\Omega$}
\end{equation}
with a positive constant $\tilde\delta$.

\begin{comment}
Taking \eqref{ceq-*2}, \eqref{ceq-*3} and \eqref{equiv-IBP} into account, consider
\begin{equation}\label{f4.3}
\left\{
\begin{array}{l}
\rho\partial_t^2u-\nabla \cdot \sigma[u,\phi]=0,\\
\eta_j\partial_t\phi_j-\sigma_j[u,\phi_j]=0,\,\,\,1\le j\le n_1,\\
u=0\quad {\rm over}\quad\Gamma_D,\\
\sigma[u,\phi]\nu=0 \quad {\rm over}\quad \Gamma_N, \\
(u,\partial_tu,e[u],\phi)|_{t=0}=(u^0,v^0,0,0)
\end{array}
\right.
\end{equation}
with $\phi:=(\phi_1,\cdots,\phi_{n_1})$.
\end{comment}

For $\phi:=(\phi_1,\cdots,\phi_n)=(\phi_1,\cdots,\phi_{n_1},0,\cdots,0)$ with $n_1=n$ and $1\le n_1<n$ in \eqref{equiv-IBP}, we define the volumetric viscous strain $\phi_j^V$ and the deviatoric viscous strain $\phi_j^D$ as
\begin{equation*}
\left\{
\begin{aligned}
\phi^V_j&:=\frac{1}{3}\sum_{k=1}^3(\phi_j)_{kk}I_3=\frac{1}{3}{\rm tr} (\phi_j) I_3,\\
\phi^D_j&:=\phi_j-\phi^V_j,
\end{aligned}
\right.
\end{equation*}
where ${\rm tr}(\phi_j)$ denotes the trace of $\phi_j$ and $I_3$ is the identity matrix of $3\times3$.
By direct computations, we have the following.
\begin{equation}\label{f4.1}
\begin{aligned}
\sigma_j[u,\phi_j]=C_j(e[u]-\phi_j)=\lambda_j(\nabla\cdot u)I_3+2\mu_je[u]-(3\lambda_j+2\mu_j)\phi^V_j-2\mu_j\phi^D_j.
\end{aligned}
\end{equation}
Combining the second equation of \eqref{equiv-IBP} and \eqref{f4.1}, we have for $1\leq j \leq n_1$ that
\begin{equation}\label{f4.2}
\begin{aligned}
\eta_j\partial_t\phi^V_j+\eta_j\partial_t\phi^D_j=\lambda_j(\nabla\cdot u)I_3+2\mu_je[u]-(3\lambda_j+2\mu_j)\phi^V_j-2\mu_j\phi^D_j.
\end{aligned}
\end{equation}
Taking a trace on the second equation of \eqref{equiv-IBP}, we have 
\begin{equation}\label{f4.3}
\begin{aligned}
\eta_j\partial_t({\rm tr} (\phi_j))=(3\lambda_j+2\mu_j)(\nabla\cdot u)-(3\lambda_j+2\mu_j){\rm tr} (\phi_j)
\end{aligned}
\end{equation}
for $1\leq j \leq n_1$. From $\phi(0)=0$, \eqref{f4.2} and \eqref{f4.3}, we obtain for $1\leq j \leq n_1$ that
\begin{equation}\label{f4.4}
\left\{
\begin{aligned}
\eta_j\partial_t\phi^V_j&=\frac{3\lambda_j+2\mu_j}{3}(\nabla\cdot u)I_3-(3\lambda_j+2\mu_j)\phi^V_j,\\
\phi^V_j(0)&=0
\end{aligned}
\right.
\end{equation}
and
\begin{equation}\label{f4.5}
\left\{
\begin{aligned}
\eta_j\partial_t\phi^D_j&=2\mu_je[u]-\frac{2\mu_j}{3}(\nabla\cdot u)I_3-2\mu_j\phi^D_j,\\
\phi^S_j(0)&=0.
\end{aligned}
\right.
\end{equation}
Then, solving $\phi^V_j$ and $\phi^D_j$ from \eqref{f4.4} and \eqref{f4.5}, we obtain that 
\begin{equation}\label{f4.6}
\left\{
\begin{aligned}
\phi^V_j&=\int_0^te^{-(3\lambda_j+2\mu_j)\eta_j^{-1}(t-s)}\,\frac{3\lambda_j+2\mu_j}{3\eta_j}(\nabla\cdot u)(s)\,ds I_3,\\
\phi^D_j&=\int_0^t\big(e^{-2\mu_j\eta_j^{-1}(t-s)}\,2\mu_j\eta_j^{-1}e[u](s)-e^{-2\mu_j\eta_j^{-1}(t-s)}\,\frac{2\mu_j}{3\eta_j}(\nabla\cdot u)(s)I_3\big)\,ds
\end{aligned}
\right.
\end{equation}
for $1\le j\le n_1$
and $\phi_j^V=\phi_j^D=0$ for $n_1+1\le j\le n$.
Hence, combining \eqref{f4.1} and \eqref{f4.6}, we obtain that
\begin{equation}\label{f4.7}
\begin{aligned}
\sigma[u,\phi]=\sum_{j=1}^n\{\lambda_j(\nabla\cdot u)I_3+2\mu_je[u]-(3\lambda_j+2\mu_j)\phi^V_j-2\mu_j\phi^D_j\}.
\end{aligned}
\end{equation}

\section{UCP for Quasi-static isotropic EMM and ESLS}\label{iso-UCP}\setcounter{equation}{0}

In this section, we prove the UCP for the quasi-static EMM and ESLS.
Based on Section \ref{I-EMM}, the quasi-static EMM and ESLS can be given as
\begin{equation*}%\label{q-eqs}
\text{div}\, \sigma=0\,\,\,\text{in $\Omega\times(0,T)$ with\, $\sigma=\sigma[u,\phi]$},
\end{equation*}
where $\sigma=\sigma[u,\phi]$ is given by \eqref{f4.7}. To avoid heavy mathematical formulae and to make the flow of the proof of the UCP easier to understand, we first consider {the quasi-static EMM with $n=1$} in \eqref{equiv-IBP}. 

If we write
the elasticity tensor $C:=C_1=(C_{pqrs})$ as
\begin{equation*}
C_{pqrs}={\lambda}\delta_{pq}\delta_{rs}+{\mu}(\delta_{qs}
\delta_{pr}+\delta_{ps}\delta_{qr})
\end{equation*}
using the Lam\'e moduli $\lambda:=\lambda_1,\,\mu:=\mu_1$, then \eqref{f4.7} is equivalent to
\begin{equation}\label{f5.1}
{\sigma:=}\sigma[u,\phi]=\lambda(\nabla\cdot u)I_3+2\mu e[u]-(3\lambda+2\mu)\phi^V-2\mu\phi^D.
\end{equation}
with 
\begin{equation}\label{f5.2}
\left\{
\begin{aligned}
\phi^V&=\int_0^te^{-(3\lambda+2\mu)\eta^{-1}(t-s)}\,\frac{3\lambda+2\mu}{3\eta }(\nabla\cdot u)(s)\,ds\, I_3,\\
\phi^D&=\int_0^te^{-2\mu\eta^{-1}(t-s)}\,2\mu\eta^{-1}e[u](s)-e^{-2\mu\eta^{-1}(t-s)}\,\frac{2\mu}{3\eta}(\nabla\cdot u)(s)I_3\,ds.
\end{aligned}
\right.
\end{equation}
Since there are two differential operators $\frac{3\lambda+2\mu}{3}(\nabla\cdot u)I_3$ and $2\mu e[u]-\frac{2\mu}{3}(\nabla\cdot u)I_3$ in \eqref{f5.2}, it is difficult to prove the UCP using equation \eqref{f5.1} as it is. Therefore, in preparation for deriving a Carleman estimate leading to the UCP, we apply the divergence operator and rotation operator to \eqref{f5.1} and increasing the number of dependent variables, we aim to obtain a new equation with the Laplace operator in its main part. To do this, we need to assume that for $\lambda,\,\mu,\,\eta$ not only
$\lambda,\,\mu,\,\eta\in L^\infty(\Omega)$, but also $\lambda,\,\mu,\,\eta\in  C^2(\overline\Omega)$ so that $u\in C^2([0,T]; H^2_{loc}(\Omega))$ by applying the coercivity of the quasi-static EMM and ESLS (see Appendix).

We start by introducing two auxiliary functions $p=\nabla\cdot u$ and $w=\nabla\times u =(w_1,w_2,w_3)^{\mathfrak{t}}$ to have the description
\begin{equation}\label{f5.3}
Ce[u]=\lambda p I_3 + 2\mu e[u],
\end{equation}
and also introduce the notation $U$ to denote $U:=(u,p,w)^{\mathfrak{t}}:\Omega \rightarrow {\mathbb R}^7$ with the components obtained from $u$. We note here that $U\in C^2([0,T];H^2_{loc}(\Omega))$ (see for $V$ in Appendix which is nothing but $U$).
\begin{lemma}\label{f_jem5.1}
We have the description
\begin{equation}\label{f5.4}
\begin{array}{ll}
{\rm diag} (I_3, \nabla\cdot , \nabla \times)\, {\rm diag} (\nabla\cdot(Ce[u]), \nabla\cdot(Ce[u]) , \nabla\cdot(Ce[u]))\\
\qquad\qquad\qquad= D\Delta U+\tilde{A}_1(\nabla_x)U,
\end{array}
\end{equation}
where $D={\rm diag}(\mu,\mu,\mu,\lambda+2\mu,\mu,\mu,\mu)$ and $\tilde{A}_1(\nabla_x)$ are first-order linear operators. Here and hereafter, {$``\rm{diag}(\,\,\,\,)"$} denotes a block diagonal form with diagonal block elements that can be scalars, vectors, tensors, or operators in the bracket, and blockwise multiplications or operations for them are done in natural ways.
\end{lemma}

\begin{proof}
By \eqref{f5.3}, we have
\begin{equation}\label{f5.5}
\nabla\cdot(Ce[u])=\mu\Delta u +(\lambda+\mu)\nabla p+(\nabla\lambda) p+(\nabla u+(\nabla u)^{\mathfrak{t}})\nabla\mu.
\end{equation}
Taking the divergence on \eqref{f5.5}, it gives 
\begin{equation}\label{f5.6}
\begin{aligned}
\nabla\cdot \left(\nabla\cdot(Ce[u])\right)=&(\lambda+2\mu)\Delta p+2\nabla\mu\cdot\Delta u
+2(\nabla(\lambda+\mu))\cdot\nabla p\\
&+(\Delta \lambda) p
+\sum_{i,j=1}^3(\partial_j u_i +\partial_i u_j )\partial^2_{ij}\mu,
\end{aligned}
\end{equation}
where $\partial_j:=\partial_{x_j}$ and $\partial_{ij}^2:=\partial_i\,\partial_j$.
By using
\begin{equation*}%\label{f5.8}
\Delta u=\nabla(\nabla\cdot u)-\nabla\times(\nabla\times u)=\nabla p-\nabla\times w
\end{equation*}
in \eqref{f5.5}, we have
\begin{equation}\label{f5.7}
\begin{aligned}
\nabla\cdot \left(\nabla\cdot(Ce[u])\right)=&(\lambda+2\mu)\Delta p+2\nabla\mu\cdot(\nabla p-\nabla\times w)
+(\Delta\lambda) p\\
&+(2\nabla(\lambda+\mu))\cdot\nabla p
+\sum_{i,j=1}^3(\partial_j u_i +\partial_i u_j )\partial^2_{ij}\mu.
\end{aligned}
\end{equation}
Similarly, applying the curl on \eqref{f5.5} and using the identity
$\nabla\times(fu)=f\nabla\times u+\nabla f\times u$ and \eqref{f5.5}, we have
\begin{equation}\label{f5.8}
\begin{array}{ll}
\nabla\times \left(\nabla\cdot(Ce[u])\right)=\mu\Delta w+\nabla\mu\times(\nabla p
-\nabla\times w)+\nabla(\lambda+\mu)\times\nabla p\\
\qquad\quad
+\nabla p\times\nabla\lambda+\sum_{j=1}^3(\partial_{j}\mu\partial_{j}w
+\nabla\partial_{j}\mu\times(\partial_{j}u+\nabla u_j)).
\end{array}
\end{equation}
 Combining  \eqref{f5.6}, \eqref{f5.7} and \eqref{f5.8}, we derive \eqref{f5.4}. 
\end{proof}

\begin{lemma}\label{f_jem5.2}
Let $a=a(x)\in C^2(\Omega)$ be a scalar function. Then, we have that
\begin{equation}\label{f5.9}
\begin{aligned}
&{\rm diag} (I_3, \nabla\cdot , \nabla \times)\, {\rm diag} (\nabla\cdot(a\, Ce[u]), \nabla\cdot(a\, Ce[u]), \nabla\cdot(a\, Ce[u]))\\
=& a D\Delta U+(G)_1(\nabla_x)U,
\end{aligned}
\end{equation}
where $(G)_1(\nabla_x)$ is a first-order linear operator.

\end{lemma}
\begin{proof} Note that
\begin{equation}\label{f5.10}
\begin{aligned}\nabla\cdot(a\,Ce[u])=a\,\nabla\cdot(Ce[u])+(Ce[u])\nabla a.
\end{aligned}
\end{equation}
Then, from \eqref{f5.5}, the term
\begin{equation*}
\begin{aligned}
{\rm diag} (I_3, \nabla\cdot , \nabla \times)\, {\rm diag} (a\,\nabla\cdot(Ce[u]), a\,\nabla\cdot(Ce[u]), a\,\nabla\cdot(Ce[u]))
\end{aligned}
\end{equation*}
has the form of \eqref{f5.4} with $D\Delta U$ replaced by $a\,D\Delta U$ and  a first-order linear operator acting on $U$.
Hence, we focus on the second-order derivative terms of $u$ in
\begin{equation}\label{f5.11}
\begin{aligned}
{\rm diag} (I_3, \nabla\cdot , \nabla \times)\, {\rm diag} ((Ce[u])\nabla a, (Ce[u])\nabla a, (Ce[u])\nabla a).
\end{aligned}
\end{equation}
From \eqref{f5.3}, we only need to compute the following two terms
$$ \nabla\cdot(e[2u]\nabla a), \qquad \nabla\times (e[2u]\nabla a).$$ 
The second-order derivative terms of $u$ in $\nabla\cdot(e[2u]\nabla a)$ and $\nabla\times (e[2u]\nabla a)$ are
\begin{equation}\label{f5.12}
\begin{aligned}
 (\nabla\cdot e[2u])\cdot\nabla a=(\Delta u+\nabla p)\cdot\nabla a=(2\nabla p-\nabla\times w )\cdot\nabla a,
\end{aligned}
\end{equation}
and
\begin{equation}\label{f5.13}
\begin{aligned}
 (\nabla\times e[2u])^\mathfrak{t}\nabla a(x)=(\nabla(\nabla \times u)) \nabla a=(\nabla w) \nabla a.
\end{aligned}
\end{equation}
From \eqref{f5.12} and \eqref{f5.13}, we can express the second-order derivative terms of $u$ in \eqref{f5.13} by the first-order derivative terms of $U$.
Hence, from \eqref{f5.4}, \eqref{f5.10} and \eqref{f5.11}, we can derive \eqref{f5.9}.
\end{proof}

In addition, we can handle the integrands of \eqref{f5.2} in the same way. Thus, from \eqref{f5.4} and \eqref{f5.9},  
\begin{equation*}%\label{diag}
{\rm diag} (I_3, \nabla\cdot , \nabla \times)\, {\rm diag} (\nabla\cdot\sigma, \nabla\cdot\sigma, \nabla\cdot\sigma)=0
\end{equation*}
has the following description
\begin{equation}\label{f5.14}
\begin{aligned}
\Delta V(x,t)-\int_0^t  D_0(x,t,s)\Delta V(x,s)\,ds\,=A_1(\nabla_x)V(x,t)+\int_0^t  B_1(\nabla_x)V(x,s)\,ds\, 
\end{aligned}
\end{equation}
for $0\leq t \leq T$, where 
\begin{comment}
$V=(v_1,\cdots,v_7)^\mathfrak{t}$,
\end{comment}
$A_1(\nabla_x)$ and
$B_1(\nabla_x)$ are the first-order linear operators, $D_0(x,t,s)={\rm diag}(d_1(x,t,s),\cdots,d_7(x,t,s))$ with each $d_j(x,t,s)=d_j(x,t-s)$ such that $d_j(x,t)\in C^0(\overline{\Omega}\times[0,T])$ satisfies the estimate $ |d_j(x,t)| \leq b_0 e^{b_1t}$ for $1\leq j\leq 7$ on $\overline{\Omega}\times[0,T]$ for some positive constants $b_0,b_1$. Here, we used $V$ instead of $U$ to write
$V=(v_1,\cdots,v_7)^\mathfrak{t}=(u,p,w)^\mathfrak{t}$ in the form, which is to avoid unnatural descriptions $p=u_4,\,w=(u_5,u_6,u_7)^{\mathfrak{t}}$. We also remark that $\Delta V(x,t)-\int_0^t  D_0(x,t,s)\Delta V(x,s)\,ds$ is decomposed into each component of $V=(v_1,\cdots,v_7)^\mathfrak{t}$. Hence, for this part, we only need to estimate $\Delta z(x,t)-\int_0^t  d(x,t,s)\Delta z(x,s)\,ds$ with $z=v_j$ and $d=d_j$ for $1\leq j\leq 7$.

\medskip
Upon having the description \eqref{f5.14}, we proceed to derive a Carleman estimate for $V$ that satisfies \eqref{f5.14}. For that, let us define a Carleman weight $h_\beta(x)$ by
\begin{equation*}%\label{weight}
h_\beta(x)=h_\beta(|x|)=\exp{\Big (\frac{\beta}{2}(\log|x-x_0|)^2\Big )},\,\,\,x\in {\mathbb R}^3\setminus\{x_0\}
\end{equation*}
with a positive parameter $\beta$. The suppressed $x_0$ becomes clear from the context in the following arguments.
In addition, to simplify the notation, $\int{\cdot}\,dx$ denotes an integral in $\Omega$ for the rest of this section.

First of all, our Carleman estimate is based on the following Carleman estimate given as Theorem 3.1 in \cite{Lin}. 

\begin{lemma}\label{f_lem5.3}
There exist a large number of $\beta_0>0$ and positive constants $a,\,r_0$ such that for any fixed $s$ with $0\leq s \leq T$ and $z(\cdot,s)\in W_{r_0}(x_0)$ with $0<r_0<e^{-1}$, $\beta>\beta_0$, we have
\begin{equation}\label{f5.15}
\beta\int h^2_\beta(x) (|\nabla z(x,s)|^2+|z(x,s)|^2)\, dx
\leq a\int h^2_\beta(x) |\Delta z(x,s)|^2\, dx,
\end{equation}
where $W_{r_0}(x_0):=\{z\in H^{2}(\Omega ): {\rm supp}(z)\subset{B_{r_0}(x_0) \setminus \{x_0\}} \subset \Omega\}$.
\end{lemma}
\begin{remark}\label{f_jem5.4}
In \cite{Lin}, \eqref{f5.15} is given for $z(x,s)\in C_{0}^{\infty}(B_{r_0}(x_0) \setminus \{x_0\})$. By cutting $z$ for a small $|x-x_0|$ and regularizing, the estimate \eqref{f5.15} still holds for $z\in W_{r_0}(x_0)$. More precisely, we can let a cut function $\chi(|x-x_0|)$ with support in ${\rm supp}(\chi)\subset{B_{r_0}(x_0)}$. Let $z_\epsilon=(\chi z)*\zeta_\epsilon$ with $\zeta$ be a mollifier and $\zeta_\epsilon=\epsilon^{-3}\zeta(x/\epsilon)$. Since $z_\epsilon\in C_{0}^{\infty}(B_{r_0}(x_0) \setminus \{x_0\})$ converges to $\chi z$ in $H^2$ norm as $\epsilon$ tends to $0$, the estimate \eqref{f5.15} still holds with a different constant $a$.
\end{remark}

The following two lemmas provide estimates for the case where the function $d$ is involved.
\begin{lemma}\label{f_lem5.5} Let $z(x,s)\in L^{\infty}([0,T]; H^2(\Omega))$ and $z(x,s)=0$ in $ B_{r_0}(x_0)\times (0,T)\subset\Omega\times(0,T)$ with $0<r_0<e^{-1}$, we have 
\begin{equation}\label{f5.16}
\begin{aligned}%{ll}
&\int_0^t \,\int h^2_\beta(x)\,\Big |\int_0^s d(x,s,\tau)\Delta z(x,\tau)  d\tau\, \Big |^2\,dxds\\
&\qquad\qquad\qquad\qquad
\leq  \,b_0^2t^2e^{2b_1t}\,\int_0^t \int h^2_\beta(x)\, |\Delta z(x,s)|^2\, dx ds,
\end{aligned}
\end{equation}
where the pair $z,\,d(x,s,\tau)$ represents the pair $v_j,\,d_j(x,t,s)$ for one of $1\leq j\leq 7$.
\end{lemma}

\begin{proof}
By just estimating $d(x,s,\tau)=d(x,s-\tau)$ from above by $b_0e^{b_1s}$ and using H\"older's inequality, Fubini's theorem, we have
% \begin{equation*}
\begin{align*}
&\int_0^t \,\int h^2_\beta(x)\,|\int_0^s d(s,\tau,x)\Delta z(x,\tau)  d\tau\, |^2\, dxds\\
\leq & \, b_0^2\int_0^t \,s\, e^{2b_1s}\int \,\int_0^s h^2_\beta(x)\,| \Delta z(x,\tau)|^2\,  d\tau dxds\\
= & \, b_0^2\int_0^t \,\int_\tau^t s\,e^{2b_1s}\,(\,\int h^2_\beta(x)\,| \Delta z(x,\tau)|^2\,  dx) ds d\tau \\
\leq & \,b_0^2t^2e^{2b_1t}\,\int_0^t \int h^2_\beta(x) |\Delta z(x,\tau)|^2 dx d\tau\\
= & \,b_0^2t^2e^{2b_1t}\,\int_0^t \int h^2_\beta(x)\, |\Delta z(x,s)|^2\, dx ds,
\end{align*}
% \end{equation*}

\begin{comment}
% \begin{equation*}
\begin{align*}
&\int_0^t \,\int h^2_\beta(x)\,|\int_0^s d(s,\tau,x)\Delta z(x,\tau)  d\tau\, |^2\, dxds\\
\leq&\, b_0^2\int_0^t \,e^{2b_1s}\int h^2_\beta(x)\,(\int_0^s | \Delta z(x,\tau)|  d\tau )^2\,dxds\\
\leq & \, b_0^2\int_0^t \,s\, e^{2b_1s}\int \,\int_0^s h^2_\beta(x)\,| \Delta z(x,\tau)|^2\,  d\tau dxds\\
= & \, b_0^2\int_0^t \,\int_0^s s\,e^{2b_1s}\,(\,\int h^2_\beta(x)\,| \Delta z(x,\tau)|^2\,  dx) d\tau ds\\
= & \, b_0^2\int_0^t \,\int_\tau^t s\,e^{2b_1s}\,(\,\int h^2_\beta(x)\,| \Delta z(x,\tau)|^2\,  dx) ds d\tau \\
\leq & \,b_0^2t^2e^{2b_1t}\,\int_0^t \int h^2_\beta(x) |\Delta z(x,\tau)|^2 dx d\tau\\
= & \,b_0^2t^2e^{2b_1t}\,\int_0^t \int h^2_\beta(x)\, |\Delta z(x,s)|^2\, dx ds,
\end{align*}
% \end{equation*}
\end{comment}
\noindent
which is nothing but \eqref{f5.16}.
\end{proof}

Using \eqref{f5.16}, we also have the following.
\begin{lemma}\label{f_lem5.6}${}$
Let $t,\,z$ be $\,0\leq t \leq T_0:=\min{({1}/{4eb_0},{\ln 2}/{b_1})}$, $z=z(x,s)\in L^{\infty}([0,T]; H^2(\Omega))$ and $z(x,s)=0$ in $ B_{r_0}(x_0)\times (0,T)\subset\Omega\times(0,T)$ with $0<r_0<e^{-1}$. Then, we have  
\begin{equation}\label{f5.17}
\begin{aligned}%{ll}
&\frac{1}{2}\int_0^t \int h^2_\beta(x)\, |\Delta z(x,s)|^2\, dx ds\\
&\leq \int_0^t \,\int h^2_\beta(x)\,\big[  \Delta z(x,s)- \int_0^s d(x,s,\tau)\Delta z(x,\tau)  d\tau \big]^2\,dxds.
\end{aligned}
\end{equation}
\end{lemma}

\begin{proof}
We put $$g(x,s)=\int_0^s d(x,s,\tau)\Delta z(x,\tau)  d\tau.$$
By using \eqref{f5.16} and H\"older's inequality, we have 
% \begin{equation*}
\begin{align*}
&\frac{1}{2}\int_0^t \int h^2_\beta(x)\, |\Delta z(x,s)|^2\, dx ds\\
=&\int_0^t \int h^2_\beta(x)\, |\Delta z(x,s)|^2\, dx ds-\frac{1}{2}\int_0^t \int h^2_\beta(x)\, |\Delta z(x,s)|^2\, dx ds\\
\leq &\int_0^t \int h^2_\beta(x)\, |\Delta z(x,s)|^2\, dx ds+\int_0^t \,\int h^2_\beta(x)\,|g(x,s)\, |^2\,dxds\\
&-2\,  \left(\int_0^t \int h^2_\beta(x)\, |\Delta z(x,s)|^2 dx ds\right)^{1/2}\cdot\left(\frac{1}{16}\int_0^t \int h^2_\beta(x)\, |\Delta z(x,s)|^2 dx ds\right)^{1/2}\\
\leq &\int_0^t \int h^2_\beta(x)\, |\Delta z(x,s)|^2\, dx ds+\int_0^t \,\int h^2_\beta(x)\,|g(x,s)\, |^2\,dxds\\
&-2\,  \left(\int_0^t \int |h_\beta(x) \Delta z(x,s)|^2 dx ds\right)^{1/2}\cdot\left(\int_0^t \,\int |h_\beta(x)g(x,s)\, |^2\,dxds\right)^{1/2}\\
\leq &\int_0^t \int h^2_\beta(x)\, |\Delta z(x,s)|^2 dx ds+\int_0^t \int h^2_\beta(x)\,|g(x,s)\, |^2 dxds\\
&-2\,  \int_0^t \int \left(h_\beta(x)\, \Delta z(x,s)\right)\cdot \left(h_\beta(x)g(x,s)\,\right) \, dx ds\\
=& \int_0^t \,\int h^2_\beta(x)\,\big[  \Delta z(x,s)- g(x,s)\, \big]^2\,dxds,
\end{align*}
% \end{equation*}
which is nothing but \eqref{f5.17}.
\end{proof}

\medskip
By summing up, we have the following theorem for the case $n=1$.

\begin{comment}
\begin{Th}\label{thm5.7}
Let {\color{red} $V\in C^0([0,T]; H^2_{loc}(\Omega))$} be a solution of equation \eqref{f5.16}   in $\Omega\times(0,T)$ and $V(x,t)=0$ in $B_R(0)\times(0,T)$ with $B_R(0)\subset \Omega$. Then $V(x,t)=0$ in $\Omega\times(0,T)$.
\end{Th}
    
\end{comment}

\begin{Th}\label{thm5.7}
Let $u\in C^2([0,T]; H^1(\Omega))$ be a solution of equation \eqref{equiv-IBP} with \eqref{f5.1}   in $\Omega\times(0,T)$ and $u(x,t)=0$ in $B_R(x_0)\times(0,T)\subset\Omega\times(0,T)$. Then $u(x,t)=0$ in $\Omega\times(0,T)$.
\end{Th}

\begin{proof}

Applying the above two lemmas to each pair $z=v_j$ and $d=d_j$ for $1\le j\le 7$, we derive a Carleman estimate for $V$.
\begin{comment}
We take $z=v_j$ and $d=d_j$ for $1\leq j\leq 7$.
For this $z\in C^0([0,T]; H^2_{loc}(\Omega))$, we further show a Carleman estimate and how we derive the UCP from the estimate by taking $x_0=0$ for simplicity of notation.
\end{comment}
To begin with, let $x_0=0$ for the simplicity of notation and take $\xi(x)\in C^\infty_0(\mathbb{R}^3)$ to be such that $\xi(x)=1$ for $|x|<R$ and $\xi(x)=0$ for $|x|>2R$ ($0<r<R<r_0/2$).
From \eqref{f5.15}, \eqref{f5.17} and \eqref{f5.14}, we have
\begin{equation}\label{f5.18}
\begin{array}{ll}
\sum_{j=1}^7\beta\int_0^t\int_{|x|<R} h^2_\beta(x) (|\nabla v_j(x,s)|^2+|v_j(x,s)|^2) dxds\\
\\
\leq \sum_{j=1}^7\beta\int_0^t\int h^2_\beta(x) (|\nabla (\xi v_j)(x,s)|^2+|\xi v_j(x,s)|^2) dxds\\
\\
\leq  \sum_{j=1}^7a\int_0^t\int h^2_\beta(x) |\Delta(\xi v_j)(x,s)|^2 dxds\\
\\
\leq  \sum_{j=1}^7 2a\int_0^t \,\int h^2_\beta(x)\,\big[  \Delta(\xi v_j )(x,s)- \int_0^s d_j(s,\tau,x)\Delta (\xi v_j)(x,\tau) d\tau\big]^2\,dxds\\
\\
\leq \sum_{j=1}^7 4a \int_0^t \,\int h^2_\beta(x)\,\big[ -\xi |A_1(\nabla_x)V(x,s)|+\xi \int_0^s \,|B_1(\nabla_x)V(x,\tau)|\,d\tau \big]^2\,dxds\\
\\
\,\,+\sum_{j=1}^7 4a\int_0^t \,\int_{|x|>R} h^2_\beta(x)\,\big[  [\Delta,\xi]v_j(x,s)- \int_0^s d_j(s,\tau,x)[\Delta,\xi]v_j(x,\tau) d\tau \big]^2\,dxds,
\end{array}
\end{equation}
where  $[\Delta,\xi]$ is the commutator of $\Delta$ and $\xi$.

From \eqref{f5.18}  and the monotone decreasing property of the weight function with respect to $|x|$, we have
\begin{equation}\label{f5.19}
\begin{aligned}
&\beta\int_0^t\int_{|x|<R} h^2_\beta(x) (|\nabla V(x,s)|^2+|V(x,s)|^2)\, dxds\\
\leq & a_1\int_0^t\int_{|x|<R} h^2_\beta(x)(|\nabla V(x,s)|^2+|V(x,s)|^2)\,dxds\\
&+ a_1\int_0^t\int_0^s\int_{|x|<R} h^2_\beta(x)(|\nabla V(x,\tau)|^2+|V(x,\tau)|^2)\,dxd\tau ds\\
&+a_1\int_0^t\int_{|x|>R} h^2_\beta(R)(|\nabla V(x,s)|^2+|V(x,s)|^2)\,dxds\\
&+ a_1\int_0^t\int_0^s\int_{|x|>R} h^2_\beta(R)(|\nabla V(x,\tau)|^2+|V(x,\tau)|^2)\,dxd\tau ds,
\end{aligned}
\end{equation}
with a positive constant $a_1$ which naturally yields from the constant $a$ in Lemma \ref{f_lem5.3} when providing the above estimates by using the fact that any norms in ${\mathbb R}^7$ are equivalent. 

Now, we estimate the second term on the right-hand side of \eqref{f5.19}.
\begin{equation}\label{f5.20}
\begin{aligned}
&a_1\int_0^t\int_0^s\int_{|x|<R} h^2_\beta(x)(|\nabla V(x,\tau)|^2+|V(x,\tau)|^2)\,dxd\tau ds\\
= &a_1\int_0^t\int_\tau^t\int_{|x|<R} h^2_\beta(x)(|\nabla V(x,\tau)|^2+|V(x,\tau)|^2)\,dx dsd\tau\\
\leq &t\cdot a_1\int_0^t\int_{|x|<R} h^2_\beta(x)(|\nabla V(x,\tau)|^2+|V(x,\tau)|^2)\, dx d\tau\\
\leq &a_1 T_0\int_0^{T_0}\int_{|x|<R} h^2_\beta(x) (|\nabla V(x,s)|^2+|V(x,s)|^2) \,dxds.
\end{aligned}
\end{equation}

Then, let $\beta>4a_1(1+T_0)$. By combining \eqref{f5.19} and \eqref{f5.20}, the first two terms on the right-hand side of \eqref{f5.19} will be absorbed by the left-hand side of \eqref{f5.19}. This leads to the following Carleman estimate.
\begin{equation}\label{f5.21}
\begin{aligned}
&\beta\cdot h^2_\beta(R)\int_0^t\int_{|x|<R}  (|\nabla V(x,s)|^2+|V(x,s)|^2) \, dxds\\
\leq&\beta\int_0^t\int_{|x|<R} h^2_\beta(x) (|\nabla V(x,s)|^2+|V(x,s)|^2) \, dxds\\
\leq 
&2a_1h^2_\beta(R)\int_0^t\int_{|x|>R} (|\nabla V(x,s)|^2+|V(x,s)|^2)\,dxds\\
&+2a_1h^2_\beta(R)\int_0^t\int_0^s\int_{|x|>R} (|\nabla V(x,\tau)|^2+|V(x,\tau)|^2)\, dxd\tau ds.
\end{aligned}
\end{equation}

The local UCP follows by letting $\beta \rightarrow \infty$ in \eqref{f5.21}. That is, $V(x,s)=0$ in $B_R(0)\times[0,T_0]$. In addition, by the standard argument, we have the global UCP given as $V(x,t)=0$ in $\Omega\times(0,T)$. 
\end{proof}

\begin{comment}
    
\begin{proof}
In the integrand of $\phi^V$, it is of the form $\gamma(x)pI$ that can be seen as a case of $Ce[u]=\lambda p I + 2\mu e[u]$ with $\lambda=\gamma$ and $\mu=0$. Similarly, we can
treat the integrand of $\phi^D$ as a case of $Ce[u]$.

Combining \eqref{f5.5} and \eqref{f5.11}, we have
\begin{equation}\label{f5.24}
\begin{aligned}
&{\rm diag} (I_3, \nabla\cdot , \nabla \times)\, {\rm diag} (\nabla\cdot\sigma, \nabla\cdot\sigma, \nabla\cdot\sigma)\\
=& D\Delta U+\tilde{A}_1(\nabla_x)U(x,t)-\int_0^t D_2\Delta U+(G)_1(\nabla_x)U(x,s)\,ds=0,
\end{aligned}
\end{equation}
where $D_2$ is a diagonal matrix and  $(G)_1(\nabla_x)$ is a first-order linear operator.

Since \eqref{f5.24} is a form of \eqref{f5.16}, we can
apply Theorem \ref{thm5.7} to conclude $u(x,t)=0$ in $\Omega\times(0,T)$.
\end{proof}

\end{comment}

Now we consider the general case $n$ in \eqref{equiv-IBP}. If we define
each elasticity tensor $C_j=((C_j)_{pqrs})$ by
$$(C_j)_{pqrs}={\lambda}_j\delta_{pq}\delta_{rs}+{\mu}_j(\delta_{qs}
\delta_{pr}+\delta_{ps}\delta_{qr})$$
and
\begin{equation}\label{f5.22}
C_je[u]=\lambda_j p I + 2\mu_j e[u],
\end{equation}
then $\sigma[u,\phi]$ in \eqref{equiv-IBP} becomes
\eqref{f4.7}.

\begin{Th}\label{f_thm5.8}
Let $u\in C^2([0,T]; H^1(\Omega))$ be a solution of equation equation \eqref{equiv-IBP} with \eqref{f5.22}   in $\Omega\times(0,T)$ and $u(x,t)=0$ in $B_R(x_0)\times(0,T)$ with $B_R(x_0)\subset \Omega$. Then $u(x,t)=0$ in $\Omega\times(0,T)$.
\end{Th}
\begin{proof}
Using the same argument given in the proof of Theorem \ref{thm5.7}, we can show that $U$ satisfies a same type of equation of the form \eqref{f5.14}. Thus, we can
apply Theorem \ref{thm5.7} to conclude $u(x,t)=0$ in $\Omega\times(0,T)$.
\end{proof}

\section{Unique identification of wave speed}\label{uniqueness}\setcounter{equation}{0}

Let the elasticity tensor $C_j=((C_j)_{pqrs})$, $\tilde{C}_j=((\tilde{C}_j)_{pqrs})$ be
\begin{equation*}%\label{f6.1}
\begin{aligned}
(C_j)_{pqrs}={\lambda}_j\delta_{pq}\delta_{rs}+{\mu}_j(\delta_{qs}
\delta_{pr}+\delta_{ps}\delta_{qr})
\end{aligned}
\end{equation*}
and
\begin{equation*}%\label{f6.2}
\begin{aligned}
(\tilde{C}_j)_{pqrs}={\tilde{\lambda}}_j\delta_{pq}\delta_{rs}+\tilde{\mu}_j(\delta_{qs}
\delta_{pr}+\delta_{ps}\delta_{qr})
\end{aligned}
\end{equation*}
for $1\leq r\leq n$.
We assume  that 
\begin{equation}\label{f6.1}
\begin{aligned}
0<\eta_j,\lambda_j, \mu_j, \tilde{\eta}_j,\tilde{\lambda}_j,\tilde{\mu}_j \in C^2(\overline\Omega)\quad {\rm for}\quad 1\leq j\leq n,\quad 0< \rho,\, \tilde{\rho}\in C^1(\overline\Omega).
\end{aligned}
\end{equation}

\begin{Th}\label{f_thm6.1}
Let $\rho,\,\tilde{\rho},\,\lambda_j,\,\tilde{\lambda}_j,\,\mu_j,\,\tilde{\mu}_j,\,\eta_j,\,\tilde{\eta}_j>0$ for $1\leq j\leq n$ satisfy \eqref{f6.1} and the strong convexity condition. Also, let $u\in C^2([0,T]; H^1(\Omega))$ be a common solution for
\begin{equation}\label{f6.2}
\left\{
\begin{array}{l}
\rho\partial_t^2u-\nabla \cdot \sigma[u,\phi]=0 \quad {\rm in}\quad \Omega\times (0,T),\\
\eta_j\partial_t\phi_j-\sigma_j[u,\phi_j]=0 \quad {\rm in}\quad \Omega\times (0,T),\\
u=g\,\,\text{over $\Gamma_D$},\,\,\sigma[u,\phi]\nu=0 \,\,\,\text{over $\Gamma_N$},\\
(u,\partial_tu,\phi)|_{t=0}=(0,0,0),
\end{array}
\right.
\end{equation}
and
\begin{equation}\label{f6.3}
\left\{
\begin{array}{l}
\tilde{\rho}\partial_t^2u-\nabla \cdot \sigma[u,\tilde{\phi}]=0 \quad {\rm in}\quad \Omega\times (0,T),\\
\tilde{\eta}_j\partial_t\tilde{\phi}_j-\sigma_j[u,\tilde{\phi}_j]=0 \quad {\rm in}\quad \Omega\times (0,T),\\
u=g\,\,\text{over $\Gamma_D$},\,\,\sigma[u,\tilde{\phi}]\nu=0 \,\,\,\text{over $\Gamma_N$},\\
(u,\partial_tu,\tilde{\phi})|_{t=0}=(0,0,0),
\end{array}
\right.
\end{equation}
where $\phi=(\phi_1,\cdots,\phi_n)$, $\tilde{\phi}=(\tilde{\phi}_1,\cdots,\tilde{\phi}_n)$.

Then, for any open subset $\omega$ compactly embedded in $\Omega$ satisfying
\begin{equation*}%\label{eq:omega}
\min_{\bar{\omega}}\left \{\left |\sum_{j=1}^n\left (\frac{\mu_j}{\rho}-\frac{\tilde{\mu}_j}{\tilde{\rho}}\right )\right |,\,\left |\sum_{j=1}^n\left (\frac{\lambda_j+2\mu_j}{\rho}-\frac{\tilde{\lambda}_j+2\tilde{\mu}_j}{\tilde{\rho}}\right )\right |\right \}>0,
\end{equation*}
we have $u(x,t)=0$ in $\omega\times (0,T)$. 
\end{Th}
\begin{proof}
By Theorem \ref{f_thm3.2}, we have $u(x,t)=0$ in
$B_R(x_0)\times (0,T_1)$ for some $B_R(x_0)\subset\omega$ and a positive constant $T_1<T$. Here, it should be noted that $T_1$ can remain the same even if we change the initial time.

From \eqref{f6.2} and \eqref{f6.3}, we have
\begin{equation}\label{f6.4}
\begin{array}{l}
 \partial_t^2 u=\rho^{-1}\nabla \cdot \left(\sum_{j=1}^n\{C_je[u]-(3\lambda_j+2\mu_j)\phi^V_j-2\mu_j\phi^D_j\} \right)
\end{array}
\end{equation}
and
\begin{equation}\label{f6.5}
\begin{array}{l}
\partial_t^2 u=\tilde{\rho}^{-1}\nabla \cdot \left(\sum_{j=1}^n\{\tilde{C}_je[u]-(3\tilde{\lambda}_j+2\tilde{\mu}_j)\tilde{\phi}^V_j-2\tilde{\mu}_j\tilde{\phi}^D_j\} \right)
\end{array}
\end{equation}
with 
\begin{equation*}
\left\{
\begin{aligned}
\tilde{\phi}_j^V&=\int_0^te^{-(3\tilde{\lambda}_j+2\tilde{\mu}_j)\tilde{\eta}_j^{-1}(t-s)}\,\frac{3\tilde{\lambda}_j+2\tilde{\mu}_j}{3\tilde{\eta}_j }(\nabla\cdot u)(s)\,ds\, I_3,\\
\tilde{\phi}_j^D&=\int_0^te^{-2\tilde{\mu}_j\tilde{\eta}_j^{-1}(t-s)}\,2\tilde{\mu}\tilde{\eta}_j^{-1}e[u](s)-e^{-2{\tilde{\mu}_j}{\tilde{\eta}_j}^{-1}(t-s)}\,\frac{2\tilde{\mu}_j}{3\tilde{\eta}_j}(\nabla\cdot u)(s)I_3\,ds.
\end{aligned}
\right.
\end{equation*}
By subtracting \eqref{f6.5} from \eqref{f6.4}, we obtain
\begin{equation*}
\begin{aligned}%{ll}
&\Big\{\rho^{-1}\nabla \cdot \big (\sum_{j=1}^n C_je[u]\big )-\tilde{\rho}^{-1}\nabla \cdot \big (\sum_{j=1}^n \tilde{C}_je[u]\big )\Big\}\\
&-\Big\{\rho^{-1}\nabla \cdot \big(\sum_{j=1}^n (3\lambda_j+2\mu_j)\phi^V_j-2\mu_j\phi^D_j \big)-\tilde{\rho}^{-1}\nabla \cdot \big(\sum_{j=1}^n (3\tilde{\lambda}_j+2\tilde{\mu}_j)\tilde{\phi}^V_j-2\tilde{\mu}_j\tilde{\phi}^D_j \big)\Big\}=0.
\end{aligned}
\end{equation*}
From \eqref{f5.5}, we have 
\begin{equation}\label{f6.6}
\begin{aligned}
&\rho^{-1}\nabla \cdot \big (\sum_{j=1}^n C_je[u]\big )-\tilde{\rho}^{-1}\nabla \cdot \big (\sum_{j=1}^n \tilde{C}_je[u]\big )\\
=&\sum_{j=1}^n(\frac{\mu_j}{\rho}-\frac{\tilde{\mu}_j}{\tilde{\rho}})\Delta u +\sum_{j=1}^n(\frac{\lambda_j+\mu_j}{\rho}-\frac{\tilde{\lambda}_j+\tilde{\mu}_j}{\tilde{\rho}})\nabla p\\
&+\sum_{j=1}^n(\rho^{-1}\nabla\lambda_j-\tilde{\rho}^{-1}\nabla\tilde{\lambda}_j) p+(\nabla u+(\nabla u)^\mathfrak{t})\sum_{j=1}^n(\rho^{-1}\nabla\mu_j-\tilde{\rho}^{-1}\nabla\tilde{\mu}_j).
\end{aligned}
\end{equation}
By Lemma \ref{f_jem5.2} and the explanation leading to \eqref{f5.14}, we can transform \eqref{f6.6} into a form of \eqref{f5.14}.
From Theorem \ref{f_thm5.8}, we have $u(x,t)=0$ in $\omega\times (0,T_1)$. Recalling what we noted at the beginning of the proof, we obtain $u(x,t)=0$ in $\omega\times (0,T)$ by repeating this argument.
\end{proof}

Theorem \ref{f_thm6.1} yields the following theorem, which is analogous to Theorem 4.2 of \cite{MY}. 

\begin{Th}\label{f_thm6.2}
Let $\rho,\,\tilde{\rho},\,\lambda_j,\,\tilde{\lambda}_j,\,\mu_j,\,\tilde{\mu}_j,\,\eta_j,\,\tilde{\eta}_j>0$ for $1\leq j\leq n$ satisfy \eqref{f6.1} and the strong convexity condition. Also, let $u\in C^2([0,T]; H^1(\Omega))$ be a common solution for \eqref{f6.2} and \eqref{f6.3}.
If $\sum_{j=1}^n(\frac{\lambda_j}{\rho}-\frac{\tilde{\lambda}_j}{\tilde{\rho}})=0$ in $\Omega$, then $\sum_{j=1}^n(\frac{\mu_j}{\rho}-\frac{\tilde{\mu}_j}{\tilde{\rho}})=0$ in $\Omega\setminus\Omega_E$,
where 
\begin{equation*}
\begin{aligned}
\Omega_E :=\cup\{V\subset\Omega\,\,\text{\rm is an open set satisfying}\, u=0\,\,\text{\rm in $V\times(0,T)$}\}.
\end{aligned}
\end{equation*}
\end{Th}

\begin{proof}
Let us decompose $\Omega$ into 
\[
\Omega = \Omega^0 \cup \Omega^1,\,\,\Omega^0\cap\Omega^1=\emptyset,
\]
\begin{equation*}
\Omega^0:= \left\{x \in \Omega:\,\sum_{j=1}^n\left(\frac{\mu_j(x)}{\rho(x)}-\frac{\tilde{\mu}_j(x)}{\tilde{\rho}(x)}\right)=0\right\}, \quad
\Omega^1:=\left\{x \in \Omega:\,\sum_{j=1}^n\left(\frac{\mu_j(x)}{\rho(x)}-\frac{\tilde{\mu}_j(x)}{\tilde{\rho}(x)}\right) \neq 0 \right\}.
\end{equation*}
For the proof of the theorem, it is enough to prove $\Omega^1 \subset \Omega_E$. For that, we note the assumption
\[
\sum_{j=1}^n\left(\frac{\lambda_j}{\rho}-\frac{\tilde{\lambda}_j}{\tilde{\rho}}\right)=0\,\,\,\text{in $\Omega$} 
\]
implies
\[
\sum_{j=1}^n\left (\frac{\mu_j}{\rho}-\frac{\tilde{\mu}_j}{\tilde{\rho}}\right )=\sum_{j=1}^n\left (\frac{\lambda_j+2\mu_j}{\rho}-\frac{\tilde{\lambda}_j+2\tilde{\mu}_j}{\tilde{\rho}}\right)\,\,\,\text{in $\Omega$}.
\]
Then, for any $x \in \Omega^1$, there exists an open set $\omega\subset\Omega^1$ with $x\in \omega$, where $\omega$ is defined as in Theorem \ref{f_thm6.1}. Hence, from Theorem \ref{f_thm6.1}, we have $u(x,t) = 0$ for all $(x,t) \in \omega \times (0,T)$. Therefore, $\omega\subset \Omega_E$, and we consequently have $\Omega^1 \subset \Omega_E$.
\end{proof}

\begin{comment}
 \begin{Th}[Unique identification of compressional wave speed]\label{f_thm6.3} Under the same hypothesis in Theorem \ref{f_thm6.2}. If $\sum_{j=1}^n{\color{red}(\frac{\mu_j}{\rho}-\frac{\tilde{\mu}_j}{\tilde{\rho}})}=0$ in $\Omega$, then $\sum_{j=1}^n{\color{red}(\frac{\lambda_j}{\rho}-\frac{\tilde{\lambda}_j}{\tilde{\rho}})}$ in $\Omega \setminus\Omega_{\color{red}E}$.

$$\Omega_{\color{red}D} :=\cup\{V\subset\Omega \,{\rm is\, an\, open\, set\, satisfying}\, \|\nabla \cdot u\|_{L^2(V\times(0,T))}=0\}.$$
\end{Th}   
\end{comment}
\begin{comment}
\begin{equation}\label{f6.7}
\begin{array}{l}
 \partial_t^2 u=\rho^{-1}\nabla \cdot \sum_{j=1}^n C_j\left(e[u]-\int_0^t e^{-(t-s)\,\eta_j^{-1}C_j} \,(\eta_j^{-1}C_j\,e[u(s)])\,ds \right)
\end{array}
\end{equation}
and
\begin{equation}\label{f6.8}
\begin{array}{l}
 \partial_t^2 u=\tilde{\rho}^{-1}\nabla \cdot \sum_{j=1}^n \tilde{C}_j\left(e[u]-\int_0^t e^{-(t-s)\,\tilde{\eta}_j^{-1}\tilde{C}_j} \,(\tilde{\eta}_j^{-1}\tilde{C}_j\,e[u(s)])\,ds \right).
\end{array}
\end{equation}

By subtracting \eqref{f6.8} from \eqref{f6.7}, we obtain
\begin{equation*}%\label{f6.9}
\begin{aligned}
&\rho^{-1}\nabla \cdot \sum_{j=1}^n C_je[u]-\tilde{\rho}^{-1}\nabla \cdot \sum_{j=1}^n \tilde{C}_je[u]\\
-&\rho^{-1}\nabla \cdot \sum_{j=1}^n C_j\left(\int_0^t e^{-(t-s)\,\eta_j^{-1}C_j} \,(\eta_j^{-1}C_j\,e[u(s)])\,ds \right)\\
+&\tilde{\rho}^{-1}\nabla \cdot \sum_{j=1}^n \tilde{C}_j\left(\int_0^t e^{-(t-s)\,\tilde{\eta}_j^{-1}\tilde{C}_j} \,(\tilde{\eta}_j^{-1}\tilde{C}_j\,e[u(s)])\,ds \right)=0.
\end{aligned}
\end{equation*}
\end{comment}

\section{Conclusion and discussion}

This paper presents new research methods and results for dynamic elastography that consider the inherent viscoelastic properties of living tissues. We simplify the stress-strain relationships, in other words, relaxation tensors, for two well-known viscoelastic models (EMM and ESLS) by decomposing viscous strain into volumetric and deviatoric strains. Thereby, we succeeded in establishing a more solid foundation for the uniqueness analysis, proving the uniqueness for dynamic elastography. This uniqueness analysis involves extending J. McLaughlin and J. Yoon's ``shrink and spread" argument for isotropic elastic media to the cases of isotropic viscoelastic EMM and ESLS. Specifically, we provide a new, concise, direct proof for the finite speed of propagation, even for the anisotropic case, and a proof given for the first time for the unique continuation of solutions corresponding to shrink and spread, respectively, for the equations describing the motion. While not necessarily difficult, the derivation of the unique continuation theorem requires the interior regularity of the solution to the initial boundary value problem of viscoelastic equations with mixed-type boundary conditions, and this interior regularity appears to be new as far as we know.

Two directions for future research are proposed. First, current uniqueness analysis relies on full-wavefield data, including shear and compression waves, to uniquely identify medium parameters, whereas current clinical dynamic elastography typically uses only shear-wave signals. Future work should design specialized boundary excitation protocols to generate or isolate shear waves only in the ROI and verify the uniqueness of identifying shear wave speed using only the shear wave data.
Second, this study is limited to isotropic viscoelastic media, while many living tissues (e.g., muscle, tendon, and cardiac tissue) exhibit distinct anisotropic mechanical properties. Extending the proposed mathematical framework to anisotropic viscoelastic media and investigating the uniqueness of wave speed identification in this setting will make the dynamic elastography model more consistent with the actual mechanical properties of living tissues.

\medskip

\section*{Appendix (interior regularity of solutions)}\label{regularity}
\setcounter{equation}{0}
\setcounter{Th}{0}
\renewcommand{\thefigure}{A.\arabic{figure}}
\renewcommand{\theequation}{A.\arabic{equation}}
\renewcommand{\theTh}{A.\arabic{Th}}
In the proof of Theorem \ref{thm5.7}, we applied Lemma \ref{f_lem5.3} and Lemma \ref{f_lem5.5} with Lemma \ref{f_lem5.6} to $V$ such that it belongs to $L^\infty([0,T];H^2(\Omega))$ whenever localized by any cutoff function in $\Omega$. Hence, we need to prove $V\in L^\infty([0,T]; H^2_{loc}(\Omega))$. This Appendix aims to provide its proof.  
We start the proof by introducing the difference quotient in the direction $e_i$ by
\[
\Delta^hu(x)=\Delta_i^hu(x)=h^{-1}\big(u(x+he_i)-u(x)\big), \quad h\neq 0.
\]
From the first equation of \eqref{equiv-IBP} and \eqref{f4.7}, we have 
\begin{equation}\label{f8.1}
\begin{aligned} 
\int_{\Omega}\sum_{j=1}^n\{\lambda_j(\nabla\cdot u)I_3+2\mu_je[u]-(3\lambda_j+2\mu_j)\phi^V_j-2\mu_j\phi^D_j\} : \nabla \zeta\,dx=\int_{\Omega}(\rho\partial_t^2u , \zeta)\,dx
\end{aligned}
\end{equation}
for $\zeta=(\zeta_1,\zeta_2,\zeta_3)^{\mathfrak{t}}\in H^1_0(\Omega)$. 
Let $2h< {\rm dist} ({\rm supp}\,\zeta, \partial \Omega)$, and replace $\zeta$ with its difference quotient $\Delta^{-h}\zeta=\Delta_k^{-h}\zeta$ for some $k$, $1\leq k\leq 3$. Then, we have
\begin{equation}\label{f8.2}
\begin{aligned} 
&\int_{\Omega}\sum_{j=1}^n\Delta^{h}\{\lambda_j(\nabla\cdot u)I_3+2\mu_je[u]-(3\lambda_j+2\mu_j)\phi^V_j-2\mu_j\phi^D_j\} : \nabla \zeta\,dx\\
=&-\int_{\Omega}\sum_{j=1}^n\{\lambda_j(\nabla\cdot u)I_3+2\mu_je[u]-(3\lambda_j+2\mu_j)\phi^V_j-2\mu_j\phi^D_j\} : \nabla \Delta^{-h}\zeta\,dx\\
=&-\int_{\Omega}(\rho\partial_t^2\Delta^{h}u , \zeta)\,dx-\int_{\Omega}(\Delta^{h}\rho)(\partial_t^2u , \zeta)\,dx.
\end{aligned}
\end{equation}
Here, note that we can write 
\begin{equation*}
\left \{
\begin{aligned} 
\sum_{j=1}^n\Delta^{h}\{\lambda_j(\nabla\cdot u)I_3+2\mu_je[u]\}=\sum_{j=1}^n\{\lambda_j\Delta^{h}(\nabla\cdot u)I_3+2\mu_j\Delta^{h}e[u]\}+l_1(x,t,h)\nabla u,\\
\\
\sum_{j=1}^n\Delta^{h}\{(3\lambda_j+2\mu_j)\phi^V_j-2\mu_j\phi^D_j\}=\int_0^tl_2(x,s,h)\Delta^{h}\nabla u(\cdot,s)+l_3(x,s,h)\nabla u(\cdot,s)\,ds,
\end{aligned}\right.
\end{equation*}
where $l_1(x,t,h)$, $l_2(x,s,h)$ and $l_3(x,s,h)$ are measurable multiplication operators uniformly bounded  with respect to either $(x,t,,h)$ and $(x,s,h)$ in $\Omega\times[0,T]\times\{h\not=0\}$.

Now, let $\Omega'\Subset \Omega$ be a domain with a smooth boundary and $2h< d'={\rm dist}(\Omega',\partial\Omega) $. Also, let $\xi\in  H^1_0(\Omega)$ be a cutoff function such that $0\leq\xi\leq 1$ in $\Omega$ and $\xi=1$ in $\Omega'$. Then, the strong convexity condition \eqref{iso_s-convex} yields that for $\zeta=\xi^2\Delta^{h}\nabla u$, 
\[
\int_{\Omega}\sum_{j=1}^n\{\lambda_j\Delta^{h}(\nabla\cdot u)I_3+2\mu_j\Delta^{h}e[u]\}: \xi^2\Delta^{h}\nabla u \,dx\geq \delta \|\xi \Delta^{h}\nabla u(\cdot,t) \|_{L^2(\Omega)}^2
\]
for some positive constant $\delta$. By Young's inequality and Schwarz's inequality, we can also conclude that
\begin{equation*}
\begin{aligned}
\Big |\int_0^t\Delta^{h}\nabla u(\cdot,s)\,ds\,\xi^2\Delta^{h}\nabla u(\cdot,t)\Big |
=&\Big |\xi \Delta^{h}\nabla u(\cdot,t)\Big |\,\Big |\int_0^t\xi\,\Delta^{h}\nabla u(\cdot,s)\,ds\Big |\\
\\
\leq&\frac{\delta}{8}|\xi\Delta^{h}\nabla u(\cdot,t)|^2+8\delta \, \Big |\int_0^t\xi\Delta^{h}\nabla u(\cdot,s)\,ds\Big |^2\\
\leq&\frac{\delta}{8}|\xi\Delta^{h}\nabla u(\cdot,t)|^2+8\delta t\int_0^t|\xi\Delta^{h}\nabla u(\cdot,s)|^2\,ds.
\end{aligned}
\end{equation*}
Then, by using \eqref{f8.1}, \eqref{f8.2}, Schwarz's inequality and Young's inequality, we have
\begin{equation*}%\label{f8.3}
\begin{aligned} 
&\|\xi \Delta^{h}\nabla u(\cdot,t) \|^2_{L^2(\Omega)}\\
\leq &a_2(1+T)\Big (\max_{0\leq t\leq T}\|\partial_t^2u (\cdot,t)\|^2_{H^1(\Omega)}+\max_{0\leq t\leq T}\|u(\cdot,t) \|^2_{H^1(\Omega)}+t\int_0^t\|\xi \Delta^{h}\nabla u(\cdot,s) \|^2_{L^2(\Omega)}ds\Big ),
\end{aligned}
\end{equation*}
where $a_2$ is a positive constant.
Further, applying Gronwall's inequality and taking the maximum over $[0,T]$, we conclude that
\begin{equation*}%\label{f8.4}
\begin{aligned} 
\max_{0\leq t\leq T}\| \Delta^{h}\nabla u(\cdot,t) \|^2_{L^2(\Omega')}
\leq\max_{0\leq t\leq T}\|\xi \Delta^{h}\nabla u(\cdot,t) \|^2_{L^2(\Omega)}
\leq a_3(1+T)
\end{aligned}
\end{equation*}
with a positive constant $a_3$.
Hence, by Lemma 7.24 of \cite{gilbarg2015elliptic} concerning the existence of a weak derivative using the difference quotient, we have obtained
$u\in L^{\infty}([0,T]; H^2(\Omega'))$ for any $\Omega'\Subset\Omega$ with a smooth boundary.

Upon having $u\in L^{\infty}([0,T]; H^2(\Omega'))$, for any $V\in L^{\infty}([0,T]; H^1(\Omega'))$ which satisfies \eqref{f5.14}, the same argument can derive $V\in  L^{\infty}([0,T]; H^2(\Omega''))$ for any $\Omega''\Subset\Omega'$ with a smooth boundary.

\bigskip
{\bf Acknowledgments:} Regarding financial support, the first author was supported by the National Natural Science Foundation of China (No.12241103), the second author was partially supported by the Ministry of Science and Technology of Taiwan, and the third author was partially supported by JSPS KAKENHI (Grant No. JP25K07076). 

\printbibliography
\begin{comment}

\end{comment}

\end{document}